\documentclass[12pt]{amsart}
\usepackage{amsthm,amssymb,amsmath,amscd,epic,eepic}
\usepackage[all]{xy}

\newcommand     {\comment}[1]   {}
\newcommand{\mute}[2] {}
\newcommand     {\printname}[1] {}
\newcommand{\labell}[1] {\label{#1}}

\numberwithin{equation}{section}
\newtheorem {Theorem}                   {Theorem}
\newtheorem*{Theorem*}                   {Theorem}
\newtheorem {refTheorem}[equation]      {Theorem}
\newtheorem {Lemma}[equation]           {Lemma}

\newtheorem {Corollary} [equation]      {Corollary}
\newtheorem* {Corollary*}                {Corollary}
\newtheorem {Proposition} [equation]    {Proposition}

\newtheorem* {Lemma*}                    {Lemma}

\theoremstyle{definition}
\newtheorem{Definition}[equation]{Definition}
\newtheorem*{Definition*}{Definition}

\theoremstyle{remark}
\newtheorem{Remark}[equation]{Remark}
\newtheorem*{Remark*}{Remark}
\newtheorem{Example}[equation]{Example}

\long\def\symbolfootnote[#1]#2{\begingroup%
\def\thefootnote{\fnsymbol{footnote}}\footnote[#1]{#2}\endgroup} 

\def    \cS             {{\mathcal S}}

\def \calT {{\mathcal T}}

\def \fU {{\mathfrak U}}
\def \fW {{\mathfrak W}}
\def \cA {{\mathcal A}}

\def \ol 	{\overline}
\def \ssminus	{\smallsetminus}
\def \eps 	{\epsilon}
\def \PhiT      {\Phi \mbox{-} T}
\def \del	{{\partial}}

\def    \inv    {^{-1}}

\def    \image  {\operatorname{image}}

\def    \Otb  {{\widetilde{\Omega}_{\operatorname{basic}}}}

\def    \R	{{\mathbb R}}
\def    \N	{{\mathbb N}}

\def    \C	{{\mathbb C}}
\def    \Z       {{\mathbb Z}}

\def    \t	{{\mathfrak t}}
\def    \h	{{\mathfrak h}}
\def    \ft	{{\mathfrak t}}
\def    \fh	{{\mathfrak h}}

\def	\calH	{{\mathcal H}}
\def    \fH	{{\calH}}

\def    \SU      {{\operatorname{SU}}}
\def    \SO      {{\operatorname{SO}}}

\def	\id 	{{\operatorname{id}}}
\def	\Aut 	{{\operatorname{Aut}}}
\def    \tPhi   {{\widetilde{\Phi}}}
\def    \constant  {\operatorname{constant}}
\def    \proj   {{\text{proj}}}

\def    \Hom      {{\operatorname{Hom}}}

\def    \DH     {Duistermaat-Heckman\ }
\def    \half   {{\frac{1}{2}}}

\def	\exc 	{{\operatorname{exc}}}

\def	\ker 	{{\operatorname{ker}}}
\def	\exp	{{\operatorname{exp}}}
\def	\d	{{\operatorname{d}}}

\def	\tall   {{\operatorname{tall}}}
\def	\short  {{\operatorname{short}}}

\def	\Phibar	{{\ol{\Phi}}}
\def    \omegabar       {{\widetilde{\omega}}}

\def    \bs     {\boldsymbol}
\def    \g      {{\mathfrak g}}

\begin{document}

\title[Classification of tall complexity one spaces]
{Classification of Hamiltonian torus actions with two dimensional 
quotients}

\author[Yael Karshon]{Yael Karshon}
\address{University of Toronto, 40 St.\ George Street,
Toronto Ontario M5S 2E4, Canada}
\email{karshon@math.toronto.edu}

\author[Susan Tolman]{Susan Tolman}
\address{Department of Mathematics, University of Illinois at Urbana-Champaign,
Urbana, IL 61801}
\email{stolman@math.uiuc.edu}


\thanks{
Yael Karshon is partially supported by an NSERC Discovery Grant.
Susan Tolman is partially supported by NSF Grant DMS \#07-07122.
}

\begin{abstract}
We construct all possible Hamiltonian torus actions 
for which all the non-empty reduced spaces are two dimensional 
(and not single points)
and the manifold is connected and compact, or, more generally, 
the moment map is proper as a map to a convex set.
This construction completes the classification of 
tall complexity one spaces.
\end{abstract}

\maketitle

\tableofcontents

{

\section{Introduction}
\labell{sec:intro}

Fix a torus $T \cong (S^1)^{\dim T}$
with Lie algebra $\t$ and dual space $\t^*$.
Let $T$ act on a symplectic manifold $(M,\omega)$ 
with moment map $\Phi \colon M \to \ft^*$, so that
\begin{equation} \labell{moment}
\iota(\xi_M) \omega = - d \left< \Phi,\xi \right> 
\quad \text{ for all } \xi \in \t ,
\end{equation}
where $\xi_M$ is the vector field on $M$ induced by $\xi$.
Assume that the $T$-action is faithful (effective)
on each connected component of $M$.
We call  $(M,\omega,\Phi)$
a \textbf{Hamiltonian $\mathbf T$-manifold}.
An \textbf{isomorphism} between two Hamiltonian $T$-manifolds is
an equivariant symplectomorphism that respects the moment maps.
The \textbf{complexity} of $(M,\omega,\Phi)$ is
the difference  $\half \dim M - \dim T$; it  is half the dimension
of the reduced space  $\Phi\inv(\alpha)/T$
at a regular value $\alpha$ in $\Phi(M)$.
Assume that $(M,\omega,\Phi)$ has complexity one;
it is \textbf{tall}
if every reduced space $\Phi\inv(\alpha)/T$ is a two dimensional
topological manifold.
If $M$ is connected,
$\calT$ is a convex open subset of $\t^*$ containing $\Phi(M)$,
and the map $\Phi \colon M \to \calT$ is proper,
then we call $(M,\omega,\Phi,\calT)$ a
\textbf{complexity one space}. 
For example, if $M$ is compact and connected
then $(M,\omega,\Phi,\ft^*)$ is a complexity one space, 
which it is tall exactly if the preimage of each vertex 
of the moment polytope $\Phi(M)$ is
a fixed surface; see Corollary~\ref{dichotomy}.

In this paper we complete our classification
of tall complexity one spaces of arbitrary dimension.
More precisely, in a previous paper \cite{globun} we defined an invariant of
a tall complexity one space called the \emph{painting}, which subsumes
two other invariants: the \emph{genus} and the \emph{skeleton}
(see page~\pageref{page:genus}).
We proved that these invariants, together with the \DH measure, 
determine the tall complexity one space up to isomorphism.
In this paper we give a necessary and sufficient condition
for a measure and a painting to arise from a tall complexity one space.

Symplectic toric manifolds (see page~\pageref{page:toric})
serve as extremely important examples in symplectic geometry, 
illuminating many different aspects of the field. 
We hope that the existence theorems of this paper
will enable complexity one spaces to serve a similar role.}
These spaces are more flexible than symplectic toric manifolds.
For example, in a paper-in-progress,
the second author uses the methods of~\cite{To}
to show that many complexity one spaces 
do not admit equivariant K\"ahler structures.
Additionally, she constructs an infinite family
of symplectic forms in a fixed cohomology class 
which are equivariantly deformation equivalent
but are not equivariantly isotopic.

Symplectic toric manifolds 
are classified by their moment images \cite{delzant}; 
see \cite{KL} for the non-compact case.
Compact connected
nonabelian complexity zero actions are determined by their moment image
and the principal isotropy subgroup;
this is the Delzant conjecture, recently proved in \cite{knop:delzant}
and \cite{losev}, following earlier work in \cite{iglesias,woodward}.

The simplest complexity one spaces, compact connected symplectic 2-manifolds
with no group action, are classified by their genus and total area~\cite{moser}.
Four dimensional compact connected complexity one spaces are classified
in~\cite{karshon:periodic} (also see~\cite{ah-ha,audin:paper,audin:book});
see Example~\ref{S1}.
Similar techniques apply to complexity one nonabelian group actions
on six manifolds~\cite{river} and to two-torus actions on five dimensional
K-contact manifolds~\cite{nozawa}.
From the algebraic geometric point of view, complexity one actions
(of possibly nonabelian groups) have been studied in
\cite[Chapter IV]{KKMS}, 
\cite{timashev1,timashev2}, and
\cite{altmann-hausen,altmann-hausen-suss,altmann-peterson,vollmert}.
Moreover, a complexity one symplectic torus action 
on a compact symplectic manifold
is Hamiltonian if and only if it has a fixed point~\cite{kim}.

Work on Hamiltonian circle actions on six manifolds, 
which have complexity two, appeared in 
\cite{huili1,huili2,mcduff:six,tolman:petrie,gonzalez,LiTol}.
Finally, classification in arbitrary complexity is feasible 
for ``centered spaces" 
(\cite[section 1]{delzant}, \cite[\S 2]{kt:gromov}, \cite{KZ}).

We begin by recalling the invariants of complexity one spaces.

Let $(M,\omega,\Phi)$ be a $2n$ dimensional 
Hamiltonian $T$-manifold.
Recall that the Liouville measure on $M$ is given by integrating
the volume form $\omega^n / n!$ with respect to the symplectic orientation
and that the \textbf{\DH measure} is the push-forward of the Liouville measure
by the moment map.
The isotropy representation at $x$ is the linear representation
of the stabilizer $\{\lambda \in T \mid \lambda \cdot x = x\}$
on the tangent space $T_xM$.
Points in the same orbit have the same stabilizer,
and their isotropy representations are linearly symplectically isomorphic;
this isomorphism class is the \textbf{isotropy representation} of the orbit.

Now assume that $(M,\omega,\Phi)$ is a tall complexity one
Hamiltonian $T$-manifold.
An orbit is \textbf{exceptional} if every nearby orbit in the same
moment fiber $\Phi\inv(\alpha)$ has a strictly smaller stabilizer.  
Let $M_\exc$ denote the set of exceptional orbits in $M/T$,
and let $M'_\exc$ denote the set of exceptional orbits of another
tall complexity one space. An \textbf{isomorphism} from $M_\exc$
to $M'_\exc$ is a homeomorphism that respects the moment maps
and sends each orbit to an orbit with the same (stabilizer and)
isotropy representation.
The \textbf{skeleton} 
\labell{page:skeleton}
of $M$ is the set $M_\exc \subset M/T$ with its
induced topology, with each orbit labeled by its isotropy representation,
and with the map $\Phibar \colon M_\exc \to \t^*$
that is induced from the moment map.

The next proposition is a slight modification of Proposition 2.2
of \cite{globun}; see Remark~\ref{rmk:errors}.

\begin{Proposition} \labell{trivialize M mod T}
Let $(M,\omega,\Phi,\calT)$ be a tall complexity one space.
There exists a closed oriented surface $\Sigma$
and a map $f \colon M/T \to \Sigma$ so that
$$(\Phibar,f) \colon M/T \to (\image \Phi) \times \Sigma$$
is a homeomorphism  and
the restriction
$f \colon \Phi\inv(\alpha)/T \to \Sigma$
is orientation preserving for each $\alpha \in \image \Phi$.
Here, $\Phibar$ is induced by the moment map.
Given two such maps $f$ and $f'$, there exists an orientation preserving
homeomorphism $\xi \colon \Sigma' \to \Sigma$ so that $f$ is homotopic
to $\xi \circ f'$ through maps which induce homeomorphisms
$M/T \to (\image \Phi) \times \Sigma$.
\end{Proposition}

The \textbf{genus} 
\labell{page:genus}
of a tall complexity one space $(M,\omega,\Phi,\calT)$
is the genus of the surface $\Phi\inv(\alpha)/T$
for $\alpha \in \image \Phi$. 
By Proposition~\ref{trivialize M mod T},
it is well defined.
A \textbf{painting} is a map $f$ from $M_\exc$ to a closed oriented surface
$\Sigma$ such that
the map
$$(\Phibar,f) \colon M_\exc \to \calT \times \Sigma$$
is one-to-one, where $M_\exc$ is the set of exceptional orbits.
Two paintings, $f \colon M_\exc \to \Sigma$ and
$f' \colon M_\exc' \to \Sigma'$,
are \textbf{equivalent} if there exist 
an isomorphism $i \colon M'_\exc \to M_\exc$ 
and an orientation preserving homeomorphism
$\xi \colon \Sigma' \to \Sigma$ such that 
$f \circ i \colon M'_\exc \to \Sigma$
and $\xi \circ f' \colon M'_\exc \to \Sigma$ are homotopic
\emph{through paintings}.
Proposition \ref{trivialize M mod T} implies that
there is a well-defined equivalence class of paintings associated to
every tall complexity one space;
just  restrict $f$ to $M_\exc$.

\begin{Remark}\labell{rmk:1to1}
The notion of a painting
is simplest when $\Phibar \colon M_\exc \to \t^*$ is one-to-one,
as in Examples~\ref{6d example} and~\ref{8d example}.
In this case,
every map from $M_\exc$ to a closed oriented surface $\Sigma$ is a painting,
and two paintings $f \colon M_\exc \to \Sigma$ and
$f' \colon M_\exc \to \Sigma'$ 
are equivalent exactly if 
there exists
an orientation preserving homeomorphism $\xi \colon \Sigma' \to \Sigma$
such that $f$ and $\xi \circ f'$ are homotopic.
\end{Remark}

In the next two examples, we construct complexity one spaces
out of \textbf{symplectic toric manifolds}, i.e.,
\labell{page:toric}
compact connected complexity zero Hamiltonian $(S^1)^n$-manifolds.
The moment image of a symplectic toric manifold
is a Delzant polytope; see Remark~\ref{delzant polytope}.
In fact, every Delzant polytope
occurs as the moment image of a symplectic toric manifold, 
and this manifold
is unique up to equivariant symplectomorphism \cite{delzant}.

\begin{Example} \labell{toric}
Let $(M,\omega,\psi)$ be a symplectic toric manifold with moment image
$\Delta = \psi(M) \subset \R^n$.  
The moment map induces a homeomorphism $\ol{\psi}
\colon M/(S^1)^n \to \Delta$.  Moreover, given $x \in \Delta$, let $F_x$ be
the smallest face of $\Delta$ containing $x$.
The stabilizer of the preimage $\psi\inv(x)$ is the
connected subgroup $H_x \subset (S^1)^n$ with Lie algebra 
$$\fh_x = \{ \xi \in \R^n \mid \langle
     \xi, y - z \rangle = 0  \mbox{ for all } y, z \in F_x \}.$$

Let $\Phi \colon M \to \R^{n-1}$ be the composition
of the moment map  $\psi$ with the projection
$\pi (x_1,\ldots,x_n) = (x_1,\ldots,x_{n-1})$.
Then $(M,\omega,\Phi,\R^{n-1})$ is a complexity one space
for the subtorus $(S^1)^{n-1} \subset (S^1)^n$.
It is tall exactly if
\begin{equation} \labell{ceiling floor}
\Delta_{\text{ceiling}} \cap \Delta_{\text{floor}} = \emptyset, 
\end{equation}
where
\begin{align*}
\Delta_{\text{ceiling}} = & \left\{ x \in \Delta \ \big| \ 
x_n \geq x_n' \text{ for all } x' \in  \pi\inv(\pi(x)) 
\right\}  
\quad \text { and } \\
\Delta_{\text{floor}} = & \left\{ x \in \Delta \ \big| \  
x_n \leq x_n' \text{ for all } x' \in  \pi\inv(\pi(x))
\right\}.
\end{align*}
Assume that~\eqref{ceiling floor} holds.

For $x \in \Delta$ such that $\pi(x)$ is in the interior of $\pi(\Delta)$,
the preimage ${\psi}\inv(x)$ is exceptional exactly if
its $(S^1)^{n-1}$ stabilizer, $H_x \cap (S^1)^{n-1}$, is non-trivial.
(This always occurs\footnote
{
More generally, $H_x \cap (S^1)^{n-1}$ is trivial exactly if
$\pi(\Z^n \cap TF_x) = \Z^{n-1}$, where $TF_x = \fh_x^\circ =
\{ \lambda (y - z) \mid \lambda \in \R \mbox{ and } y, z \in F_x \}.$
To see this, note that $\pi(\Z^n \cap TF_x) = \Z^{n-1}$ exactly if every 
character of $(S^1)^{n-1}$ is the restriction of a character of $(S^1)^n$
that vanishes on $H_x$.  If $H_x \cap (S^1)^{n-1}$ is not trivial, then
there are characters of $(S^1)^{n-1}$ that don't vanish on 
$H_x \cap (S^1)^{n-1}$, and these can't be 
the restrictions of characters 
that vanish on $H_x$.
On the other hand, if $H_x \cap (S^1)^{n-1} = \{1\}$, then either 
$H_x = \{1\}$ or $(S^1)^n \simeq H_x \times (S^1)^{n-1}$. In either case,
every character of $(S^1)^{n-1}$ is the restriction of a character of
$(S^1)^n$ that vanishes on $H_x$.
}
if $\dim F_x < n -1$.)
The skeleton $M_\exc$ is
the closure of the set of such orbits.  The genus of $(M,\omega,\Phi)$ is zero.
The equivalence class of paintings associated to $M$
includes the paintings that are constant on each component of $M_\exc$.
Finally, the \DH measure is the push-forward to $\R^{n-1}$ 
of Lebesgue measure on $\Delta$.
\end{Example}

\begin{Example} \labell{associated}
Let $P \to \Sigma$ be a principal $(S^1)^n$ bundle over a closed oriented
surface $\Sigma$ of genus $g$ with first Chern class
$c_1(P) \in H^2(\Sigma,\Z^n)$, and
let $N$ be a symplectic toric manifold with moment  map
$\psi \colon N \to  \R^n$.
There exists a $T$-invariant symplectic
form $\omega$ on $M := P \times_T N$ whose restriction to the fibers is the
symplectic form on $N$;
the moment map $\Phi \colon M \to \R^n$
is given by $\Phi([p,n]) = \psi(n)$.  See \cite{GLS}.
In this case, $(M,\omega,\Phi,\R^n)$ is a
tall complexity space of genus $g$,  $M_\exc = \emptyset$,
and the \DH measure of $(M,\omega,\Phi)$ is Lebesgue measure on $\Delta$
times an  affine function with slope $c_1(P)$.
\end{Example}

\begin{Example} \labell{S1}
In \cite{karshon:periodic}, Karshon showed that a Hamiltonian circle action on a compact, connected  symplectic four-manifold $M$
is determined up to isomorphism by the following labelled graph:
The vertices correspond to connected components of the fixed point set;
each vertex is labelled by its moment map value, and -- if the 
corresponding component is a two-dimensional 
fixed surface -- the genus and area of that surface.
The edges correspond to two spheres in $M$ that the circle rotates
at speed $k > 1$; such an edge is labelled by the integer $k$.

The space $M$ is tall exactly if the minimum and maximum of the moment map
are attained along two dimensional fixed surfaces; 
see Corollary \ref{dichotomy}.
It is fairly straightforward to check that in this case the invariants 
that we describe in this paper determine, and are determined by,
the labelled graph described above.
In particular, the moment map identifies each component of the skeleton 
with an interval, and so every painting is trivial, i.e., 
equivalent to a painting that is locally constant.
\end{Example}

By  Theorem~1 of~\cite{globun}, the invariants that we have defined
completely determine the tall complexity one space:

\begin{refTheorem}[Global uniqueness] \labell{main-theorem-globun}
Let $(M,\omega,\Phi,\calT)$ and $(M',\omega',\Phi',\calT)$
be tall complexity one spaces.
They are isomorphic if and only if they have the same
moment image\footnote{
Since the moment image is the support of the \DH measure,
we could omit the condition that the spaces have the same moment image.
Nevertheless, we will sometimes include this condition for emphasis.
} and
\DH measure, the same genus, and equivalent paintings.
\end{refTheorem}

\begin{Remark} \labell{rmk:errors}
In our definition of ``equivalent paintings",
we require the homeomorphism $\xi$ to be orientation preserving.
This requirement, which is necessary for Theorem~\ref{main-theorem-globun}
to be true, was mistakenly omitted from~\cite[p.29]{globun}.
Similarly, Definition \ref{abstract painting}
of the current paper
is the correction to Definition~18.1 of~\cite{globun}.
Finally, both \cite[Proposition~2.2]{globun} 
and its smooth analogue, \cite[Lemma 18.4]{globun}, should state
that $f|_{\Phi\inv(\alpha)/T}$ and $\xi$ are orientation preserving.
The maps that are obtained in the proofs of these propositions
in \cite{globun} do satisfy this additional requirement.
\end{Remark}

Before stating our most general existence theorem,
Theorem~\ref{existence}, we give two existence theorems --
Theorems~\ref{data from manifold} and~\ref{thm:existence} -- that are
easier to state 
and simpler to apply but are sufficient for 
constructing interesting examples.
These two theorems are actually consequences 
of the most general theorem; 
all three are proved in Section~\ref{sec:proof};
cf.\ Remark~\ref{rk:special}.

\subsection*{The simplest existence theorem}

We now state our first existence theorem,
which shows that --
given a tall complexity one space  --
we can find another tall complexity one space 
with an isomorphic skeleton (and the same \DH measure)
but a different genus and painting.

\begin{Theorem} \labell{data from manifold}
Let $(M,\omega,\Phi,\calT)$ be a tall complexity one space.
Let $\Sigma$ be a closed oriented surface, and
let $f \colon M_\exc \to \Sigma$ be any painting.
Then there exists a tall complexity one space $(M',\omega',\Phi',\calT)$
with the same moment image and \DH measure whose painting is equivalent to $f$.
\end{Theorem}

\begin{Remark} \labell{rk:special}
The special case of Theorem~\ref{data from manifold}
where the genus of $\Sigma$ is equal to the genus of $M$
is easier to prove than the general case; see Theorem~\ref{reconstruction}.
\end{Remark}

\begin{Example} \labell{6d example}
Let $(M,\omega,\psi)$ be a six-dimensional symplectic toric manifold 
with moment image
$$
\Delta = \big\{  (x,y,z) \in [-3,3] \times [-2,2] \times [1,4] \ 
\big| \ |x| \leq z  \ \mbox{and} \ \ |y| \leq z \big\}.
$$
Let $\Phi \colon M \to \R^2$ be the composition 
of $\psi \colon M \to \R^3$ with the projection 
$(x,y,z) \mapsto (x,y)$.
Then $(M,\omega,\Phi,\R^2)$ is a tall complexity one space,
and $\Phi$ induces a homeomorphism from the skeleton $M_\exc$ to
the set
\begin{multline*}
\Phi(M_\exc) =  \big\{ (x,y)  \in \R^2  \, \big| \,  
|x| \leq  |y| = 1, \text{ or } |y| \leq  |x| = 1, \\
\text{ or }  1 \leq |x|=|y| \leq 2\big\};
\end{multline*}
thus, $M_\exc$ is homotopy equivalent to $S^1$.
(See Example~\ref{toric}.)

Fix a closed oriented surface $\Sigma$.
Let $[S^1,\Sigma]$ denote the set of homotopy classes of loops in $\Sigma$.
By Remark~\ref{rmk:1to1}, 
there is a one-to-one correspondence between equivalence
classes of paintings $M_\exc \to \Sigma$
and the quotient of $[S^1,\Sigma]$ 
by the action of the group $\Aut(\Sigma)$ 
of orientation preserving homeomorphisms of $\Sigma$.
If $\Sigma$ has genus zero, then since $\Sigma$ is simply connected
any two paintings are equivalent.
In contrast, if $\Sigma$ has positive genus, then the quotient
of $[S^1,\Sigma]$ by $\Aut(\Sigma)$ is infinite.  For example,
if $\Sigma$ has genus one, then this quotient is naturally isomorphic 
to the set of non-negative integers.

Therefore, if $g  = 0$ then Theorem~\ref{main-theorem-globun} implies that 
every tall complexity one space of genus $g$
whose skeleton is isomorphic to $M_\exc$
and whose \DH measure is equal to that of $(M,\omega,\Phi)$
is isomorphic to $(M,\omega,\Phi)$.
In contrast, if $g > 0$ then
Theorems~\ref{main-theorem-globun} and~\ref{data from manifold} imply that
there exist infinitely many  non-isomorphic tall complexity one spaces
with these properties.
\end{Example}

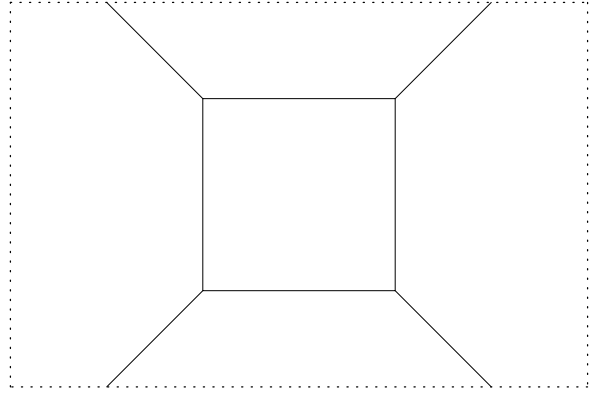
\begin{figure}[ht]
\setlength{\unitlength}{0.00083333in}
\begingroup\makeatletter\ifx\SetFigFont\undefined%
\gdef\SetFigFont#1#2#3#4#5{%
\reset@font\fontsize{#1}{#2pt}%
\fontfamily{#3}\fontseries{#4}\fontshape{#5}%
\selectfont}%
\fi\endgroup%
{\renewcommand{\dashlinestretch}{30}
\begin{picture}(3624,2439)(0,-10)
\path(1212,1812)(2412,1812)(2412,612)
(1212,612)(1212,1812)
\path(2412,1812)(3012,2412)
\path(2412,612)(3012,12)
\path(1212,612)(612,12)
\path(1212,1812)(612,2412)
\dottedline{60}(12,2412)(3612,2412)(3612,12)
(12,12)(12,2412)
\end{picture}
}
\caption{Moment image and skeleton for Example~\ref{6d example}}
\end{figure}

\begin{Example} \labell{8d example}
Fix an integer  $n > 1$.
Let $(M,\omega,\psi)$ be a $(2n+4)$-dimensional
symplectic toric manifold with moment image 
\begin{multline*}
\Delta' = \big\{  (x, y_1,\dots,y_n,z) \in 
[-3,3] \times [-2,2]^n \times [1,4] \ \big| \  \\
\ |x| \leq z 
\ \mbox{and} \ 
|y_i| \leq z  \, \text{ for all } i = 1, \ldots, n \big\}.
\end{multline*}
Composing the moment map $\psi$ with the projection $(x,y_1,\ldots,y_n,z)
\mapsto (x,y_1,\ldots,y_n)$, we obtain a tall complexity one space 
$(M,\omega,\Phi,\R^{n+1})$
such that $M_\exc$ is homotopy equivalent to $S^n$
and $\Phibar \colon M_\exc \to \R^{n+1}$ is one-to-one.
Moreover, 
the group of orientation preserving homeomorphisms acts trivially 
on the set $[S^n,\Sigma]$ of homotopy classes of maps from $S^n$ to $\Sigma$
if $\Sigma$ is an oriented surface of genus $0$,
while $[S^n,\Sigma]$ itself is trivial if  $\Sigma$ has positive genus.
Therefore, if $g > 0$ then Theorem~\ref{main-theorem-globun}
implies that every complexity one space of genus $g$
whose skeleton is isomorphic to $M_\exc$
and whose \DH measure is equal to that of $(M,\omega,\Phi)$
is isomorphic to $(M,\omega,\Phi)$.
In contrast, if $g = 0$ then Theorems~\ref{main-theorem-globun}
and~\ref{data from manifold} 
give a bijection between the set of isomorphism classes
of tall complexity one spaces with these properties
and the set $[S^n,S^2]$.
Thus, there are infinitely many non-isomorphic such spaces if $n=2$ or $n=3$.
\end{Example}

\begin{figure}[ht]
\setlength{\unitlength}{0.00083333in}
\begingroup\makeatletter\ifx\SetFigFont\undefined%
\gdef\SetFigFont#1#2#3#4#5{%
\reset@font\fontsize{#1}{#2pt}%
\fontfamily{#3}\fontseries{#4}\fontshape{#5}%
\selectfont}%
\fi\endgroup%
{\renewcommand{\dashlinestretch}{30}
\begin{picture}(3924,2439)(0,-10)
\path(1062,1662)(2262,1662)(2262,462)
(1062,462)(1062,1662)
\path(1062,1662)(1662,1962)(2862,1962)
(2862,762)(2262,462)
\path(2262,1662)(2862,1962)
\dashline{60.000}(1662,1962)(1662,762)(2862,762)
\dashline{60.000}(1062,462)(1662,762)
\dottedline{45}(12,1812)(2712,1812)(3912,2412)
(1212,2412)(12,1812)
\dottedline{45}(12,12)(2712,12)(3912,612)
(1212,612)(12,12)
\path(1662,1962)(1587,2412)
\path(2862,1962)(3537,2412)
\path(2262,1662)(2337,1812)
\path(1062,1662)(387,1812)
\path(1062,462)(387,12)
\path(2262,462)(2337,12)
\path(2862,762)(3537,612)
\path(1662,762)(1587,612)
\dottedline{45}(12,1812)(12,12)
\dottedline{45}(2712,1812)(2712,12)
\dottedline{45}(3912,2412)(3912,612)
\dottedline{45}(1212,2412)(1212,612)
\end{picture}
}
\caption{Moment image and skeleton for Example~\ref{8d example}}
\end{figure}
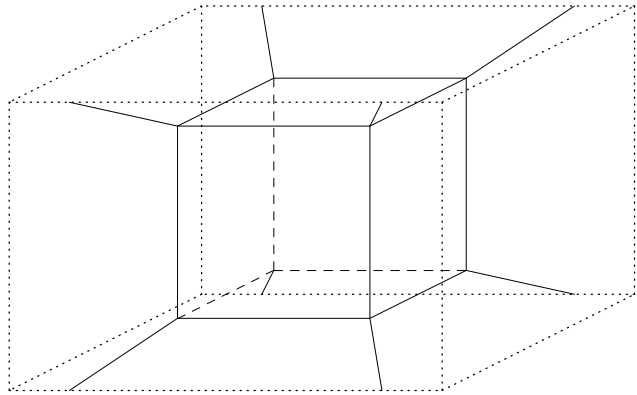

\subsection*{The intermediate existence theorem}

Our second existence theorem,
Theorem~\ref{thm:existence}, allows us to construct
complexity one spaces with prescribed painting and moment image,
even when the skeleton does not a-priori come from a complexity one space.
To state this theorem, we need an abstract notion of ``skeleton''.
To apply this theorem, one must check that the skeleton
and moment image satisfy   certain conditions. 
These conditions are automatically satisfied whenever the
the skeleton and moment image 
can be obtained from complexity one spaces over sufficiently small open sets 
in $\ft^*$; see Lemma~\ref{compatible}.
This
allows us to construct new examples by performing surgery
that attaches pieces of \emph{different} complexity one manifolds.
Such surgeries were carried out in~\cite{To, Mo};
this theorem gives a systematic way to perform such surgeries.

\begin{Definition} \labell{def:skeleton}
A \textbf{tall skeleton} over an open subset $\calT$ of $\ft^*$ is a
topological space $S$ whose points are labeled by
(equivalence classes of) representations of
subgroups of $T$, together with a proper map $\pi \colon S \to \calT$.
This data must be locally modeled on the set of exceptional orbits
of a tall complexity one space in the following sense.
For each point $s \in S$, there exists a tall complexity
one Hamiltonian $T$-manifold  $(M,\omega,\Phi)$ with
exceptional orbits $M_\exc \subset M/T$,
and a homeomorphism $\Psi$ from a neighbourhood of $s$ to an open
subset of $M_\exc$ that respects the labels and such
that $\ol\Phi \circ \Psi = \pi$,
where $\ol\Phi \colon M_\exc \to \ft^*$ is induced from the moment map.
An \textbf{isomorphism} between tall skeletons $(S',\pi')$ and
$(S,\pi)$ is a homeomorphism $i \colon S' \to S$
that sends each point to a point with the same isotropy representation
and such that $\pi'  = \pi \circ i$;
cf.\ \cite[p.\ 72]{globun}.
\end{Definition}

\begin{Remark}
In \cite[Definition 16.1]{globun} we called this notion ``skeleton".
Here we added the adjective ``tall" in order to later allow
for skeletons that are not tall.
\end{Remark}

\begin{Example} \labell{Mexc is skeleton}
If $(M,\omega,\Phi,\calT)$ is a tall complexity one space,
the set $M_\exc$, labeled with the isotropy representations,
together with the map $\Phibar$ that is
induced by the moment map, is a tall skeleton over $\calT$;
see Lemma~\ref{exceptional is closed}.
\end{Example}

\begin{Definition} \labell{abstract painting}
Let  $(S,\pi)$ be a tall skeleton over an open set $\calT \subset \t^*$
and let $\Sigma$ be a closed oriented surface.
A \textbf{painting} is a  map $f \colon S \to \Sigma$
such that the map $(\pi,f) \colon S \to  \calT \times \Sigma$ is one-to-one.
Paintings $f \colon S \to \Sigma$ and $f'\colon S' \to \Sigma'$ are
\textbf{equivalent} if there exists an isomorphism
$i \colon S' \to S$ and an orientation preserving homeomorphism
$\xi \colon \Sigma' \to \Sigma$
such that $f \circ i \colon S' \to \Sigma$
and $\xi \circ f' \colon S' \to \Sigma$ are homotopic through paintings.
\end{Definition}

The notions of painting and of equivalence of paintings
given in Definition \ref{abstract painting}
are consistent with our earlier definitions,
which only applied to the special case $(S,\pi) = (M_\exc,\Phibar)$.

Let $\ell$ denote the integral lattice in $\t$ 
and $\ell^*$ the weight lattice in $\t^*$.
Thus, $\ell = \ker(\exp \colon \t \to T)$
and $\ell^* \cong \Hom(T,S^1)$.
Here, the Lie algebra of $S^1$ is identified with $\R$
by setting the exponential map $\R \to S^1$ to be 
$t \mapsto e^{2 \pi i t}$.
Let $\R_+$ denote the set of non-negative numbers.

\begin{Definition} \labell{def:delzant}
A subset $C \subset \t^*$ is a \textbf{Delzant cone}\footnote{
Such a set is also called a ``unimodular cone".}
at $\alpha \in \t^*$
if there exist an integer $0 \leq k \leq n$
and a linear isomorphism $A \colon \R^n \to \t^*$
that sends $\Z^n$ onto the weight lattice $\ell^*$, such that
$$ C = \alpha + A(\R_+^k \times \R^{n-k}) .$$
Let $\calT$ be an  open subset of $\t^*$.
A subset $\Delta \subset \calT$ is a \textbf{Delzant subset} 
if it is closed in $\calT$ and if for every point $\alpha \in \Delta$ 
there exist a neighbourhood
$U \subset \calT$ and a Delzant cone $C$ at $\alpha$
such that $\Delta \cap U = C \cap U$.
\end{Definition}

\begin{Remark} \labell{delzant polytope}
A compact convex set $\Delta \subset \t^*$ 
is a Delzant subset exactly if it is
a \textbf{Delzant polytope},
i.e., a convex polytope such that at each vertex 
the edge vectors are generated by a basis to the lattice.
\end{Remark}

\begin{Remark} \labell{Delzant is convex}
If $\Delta$ is a Delzant subset of a \emph{convex} open subset 
$\calT \subset \t^*$ then, by the Tietze-Nakajima theorem~\cite{Tietze,N},
$\Delta$ is convex exactly if it is connected;
see~\cite{BK}.
\end{Remark}

\begin{Definition} \labell{moment cone} \ 
The \textbf{moment cone} corresponding to a point $s$
in a tall skeleton $(S,\pi)$ 
is the cone
$$ C_s := \pi(s) + (i_H^*)\inv \left( \image \Phi_s \right)\mbox{ in } \t^*,$$
where the label associated to $s$ is 
a linear symplectic representation 
of the subgroup $H$ of $T$ with quadratic moment map $\Phi_s$,
and where $i_H^* \colon \t^* \to \h^*$ is the natural projection map.
It is straightforward to check that $C_s$ is the moment image of the  
complexity one model corresponding to $s$; see Definition~\ref{def:model}.
\end{Definition}

\begin{Definition} \labell{def:compatible}
Let $\calT$ be an open subset of $\t^*$.
A Delzant subset $\Delta$ of $\calT$
and a tall skeleton $(S,\pi)$ over $\calT$
are \textbf{compatible}
if for every point $s \in S$
there exists a neighbourhood $U$ of $\pi(s)$ in $\calT$
such that $U \cap \Delta = U \cap C_s$,
where $C_s$ is the moment cone corresponding to $s$.
\end{Definition}

Let $(M,\omega,\Phi,\calT)$ be a tall complexity one space.
Then its moment map image is a convex Delzant subset of $\calT$
that is compatible with the skeleton $(M_\exc,\Phibar)$;
see Lemma~\ref{compatible}.
Our next theorem shows that this compatibility condition is also
sufficient for a subset of $\calT$ and a painting to arise from
a complexity one space.

\begin{Theorem}    \labell{thm:existence}
Let $(S,\pi)$ be a tall skeleton over a convex open subset
$\calT \subset \t^*$. Let $\Delta \subset \calT$ be a 
convex Delzant subset that is compatible with $(S,\pi)$.
Let $\Sigma$ be a closed oriented surface,
and let $f \colon S \to \Sigma$ be a painting.
Then there exists a tall complexity one space $(M,\omega,\Phi,\calT)$
with moment map image $\Delta$ whose associated painting
is equivalent to~$f$.
\end{Theorem}

\subsection*{The most general existence theorem}

Our final existence theorem, Theorem~\ref{existence},
provides a complete list of all the possible values of the invariants of
tall complexity one spaces.  Together with Theorem~\ref{main-theorem-globun},
this gives a complete classification of tall complexity one spaces.

The \textbf{\DH function} of a Hamiltonian $T$-manifold 
is a real valued function on the moment image
whose product with Lebesgue measure is equal to the \DH measure. 
If such a function exists, then it is almost unique;
any two such functions are equal almost everywhere. \labell{page:unique}
When we say that the \DH function of a Hamiltonian $T$-manifold
has some property (e.g., continuity),
we mean that this holds after possibly changing the function
on a set of measure zero. 
Here, we normalize Lebesgue measure on $\t^*$ such that the volume
of the quotient $\t^*/\ell^*$ is one.

A function $\rho \colon \t^* \to \R$
is \textbf{integral affine} if it has the form
$$ \rho(x) = \left< x, A \right> + B ,$$
where $A$ is an element of the integral lattice $\ell \subset \t$,
where $B \in \R$, and where
$\left< \cdot, \cdot \right>$ is the pairing
between $\t^*$ and $\t$.
The \DH theorem implies that 
the \DH function of a complexity one space 
with no exceptional orbits is integral affine.

Once and for all, fix an inner product on $\ft$.
Let a closed subgroup $H \subset T$ act on $\C^n$
as a subgroup of $(S^1)^n$ with quadratic moment map
$\Phi_H  \colon \C^n \to \h^*$.
Let $\fh^0 \subset \ft^*$ be the annihilator of the Lie algebra $\fh$,
and consider the model
$$ Y = T \times_H \C^n \times \h^0 ,$$
where $[ta,z,\nu] = [t,az,\nu]$ 
for all $(t,z,\nu) \in T \times \C^n \times \h^0$ and $a \in H$.
There exists a $T$  invariant symplectic form on $Y$
with moment map
$$\Phi_Y ([t,z,\nu]) = \alpha + \Phi_H(z) + \nu,$$
where $\alpha \in \t^*$ and
where we use the inner product to embed $\fh^*$ in $\ft^*$.
The isotropy representation of the orbit $\{ [t,0,0] \}$
determines the model up to permutation of the coordinates in $\C^n$.
If $\dim T = \half \dim Y - 1$, or, equivalently, $n = h+1$
where $h = \dim H$,
we call the space $Y$ a \textbf{complexity one model}.

\begin{Definition}\labell{def:model}
Given a point $s$ in a tall skeleton $S$, the \textbf{corresponding model}
is the model
$Y = T \times_H \C^{h+1} \times \h^0$ such that $s$ is labeled
by the isotropy representation of $\{ [t,0,0] \}$ in $Y$.
\end{Definition}

Such a model exists and is unique
up to permutation of the coordinates in $\C^{h+1}$.
Moreover, by Corollary~\ref{tall models}, the corresponding model 
is always  tall.

Let $Y$ be a tall complexity one model.
In Section~\ref{sec:DH for model} we define 
the \emph{ \DH functions for truncations of the model}.
(In fact, such functions are the \DH functions
of compact spaces that are obtained from $Y$ 
by extending the action to a toric action, choosing a subcircle
that is complementary to the original action,
and taking a symplectic cut with respect to this circle.)

\begin{Definition} \labell{compatible rho}
Let $(S,\pi)$ be a tall skeleton over an open subset $\calT$ of $\t^*$.
Let $\Delta \subset \calT$ be a convex Delzant subset
that is compatible with $(S,\pi)$.  Fix a point $\alpha \in \Delta$.
A function $\rho \colon \Delta \to \R_{>0}$
is \textbf{compatible with $\mathbf{(S,\pi)}$  
at the point $\mathbf{\alpha \in \Delta}$ }
if there exist for each $s \in \pi\inv(\alpha)$
a \DH function $\rho_s$ for a truncation of the tall complexity one 
model associated to $s$ such that the difference 
\begin{equation}\labell{rho}
\rho - \sum_{s \in \pi\inv(\alpha)} \rho_s
\end{equation}
is integral affine on some neighbourhood of $\alpha$ in $\Delta$.
(In particular, if $\pi\inv(\alpha)$ is empty,
then the condition is that $\rho$ itself be integral affine near $\alpha$.)
The function $\rho$ is 
\textbf{compatible with $\mathbf{(S,\pi)}$  }
if it is compatible with $(S,\pi)$ 
at every $\alpha \in \Delta$.
\end{Definition}

\begin{Remark} \labell{rk on compatible rho} \ 
The above notion of ``compatible" is in fact well defined;
moreover, the difference between any two compatible functions
is integral affine near $\alpha$.
To see this, let
$(S,\pi)$ be a tall skeleton over $\calT$;   fix $\alpha \in \calT$.
The preimage $\pi\inv(\alpha) \subset S$ is finite;
see Corollary~\ref{discrete}.
Thus, the summation in \eqref{rho} is finite.
By Corollary~\ref{exists function}, 
for each $s \in \pi\inv(\alpha)$,
there exists a \DH function $\rho_s$  for a truncation of the tall complexity
one model $Y_s$ associated to $s$;  moreover, $\rho_s$ is defined on a 
neighborhood of $\alpha$ in $\image \Phi_{Y_s}$.
By Definitions~\ref{moment cone}~and~\ref{def:compatible}, 
the moment cone $C_s = \image \Phi_{Y_s}$ coincides with $\Delta$ 
near $\alpha$
for all $s \in \pi^{-1}(\alpha)$. Thus, 
the function  in \eqref{rho}
is defined on a neighbourhood of $\alpha$ in $\Delta$.
Finally, if both $\rho_s$ and $\rho'_s$ are \DH functions 
for truncations of the model $Y_s$, then 
by Corollary~\ref{difference is integral affine}
there exists a neighbourhood of $\alpha$ in $C_s$, hence in $\Delta$,
on which 
the difference $\rho_s - \rho'_s$ coincides
with an integral affine function.
\end{Remark}

The \DH function of a tall complexity one space is compatible
with the skeleton;
see Proposition~\ref{DH is compatible}.
Our final theorem shows that this compatibility condition is
also sufficient for a function, a subset of $\calT$, and
a painting to arise from a complexity one space.

\begin{Theorem}[Global existence] \labell{existence}
Let $(S,\pi)$ be a tall skeleton over a convex open subset
$\calT \subset \t^*$, let $\Delta \subset \calT$ be a 
convex Delzant subset
that is compatible with $(S,\pi)$, and
let $\rho \colon \Delta \to \R_{>0}$ be a function that is compatible
with $(S,\pi)$. 
Let $\Sigma$ be a closed oriented surface, and let
$f \colon S \to \Sigma$ be a painting.
Then there exists a tall complexity one space over $\calT$
with moment image $\Delta$ and \DH function $\rho$ whose painting 
is equivalent to $f$.
\end{Theorem}

\smallskip

Section~\ref{sec:background} contains some general facts 
about complexity one spaces.
The remainder of the paper is divided into two parts.
Sections~\ref{sec:topology} through \ref{sec:proof reconstruction}
constitute Part I of the paper and lead to Theorem~\ref{reconstruction}.
This is a \emph{reconstruction theorem} in the sense that
we take a tall complexity one space, 
break it into pieces, 
and glue the pieces together so as to obtain a new complexity one space.
In Section \ref{sec:topology}, we prove some facts
about the cohomology of spaces that are locally
modeled on the quotients of complexity one spaces.
In Section \ref{sec:lifting}, we glue local pieces 
of complexity one spaces as $T$-manifolds.  In section \ref{sec:gluing}, 
we show how to arrange that the symplectic forms on these local pieces 
will agree on their overlaps.
In Section \ref{sec:proof reconstruction}, we use the technology developed 
so far and a crucial technical result from our previous paper
\cite[Prop.~20.1]{globun}
to prove Theorem  \ref{reconstruction}.
Sections~\ref{sec:compatible skeleton} through \ref{sec:proof}
constitute Part II of the paper.
In Section \ref{sec:compatible skeleton} we show that
the moment map image and  skeleton of a tall complexity one space
satisfy our compatibility conditions.
In Section \ref{sec:DH compatible} 
we use technical results from Section~\ref{sec:DH for model}
to show that
the \DH measure of a complexity one space is compatible with its skeleton,
and we give a local existence theorem: 
any compatible data \emph{locally} comes from a complexity one space.  
Finally, in Section \ref{sec:proof}, 
we combine these results 
with a variant of the reconstruction theorem from 
Section~\ref{sec:proof reconstruction}
to prove the main existence theorems:
Theorems \ref{data from manifold}, \ref{thm:existence}, and \ref{existence}.

\section{Basic properties of complexity one spaces}
\labell{sec:background}

In this section we recall the local normal form theorem
and the convexity package,
and analyze some of their basic  consequences 
for complexity one Hamiltonian torus actions.

\subsection*{Local normal form theorem} \

For every orbit $x$ in a Hamiltonian $T$-manifold $M$
there is a corresponding model $Y = T \times_H \C^n \times \h^0$ 
such that the isotropy representation
of the orbit $\{ [t,0,0] \} $ is the same as that of $x$.
The \textbf{local normal form for Hamiltonian torus actions}
asserts that there exists
an equivariant symplectomorphism from an invariant neighbourhood
of $x$ in $M$ to an invariant open subset of $Y$
that carries $x$ to $\{ [t,0,0] \}$; see \cite{GS:normal,marle}.

\subsection*{Convexity package} \

Let $(M,\omega,\Phi)$ be a connected Hamiltonian $T$-manifold.
Suppose that there exists a convex open subset $\calT$ of $\t^*$
that contains $\Phi(M)$ and such that $\Phi \colon M \to \calT$
is proper. Then we have the following \emph{convexity package}.

\begin{description}
\item[Convexity] The moment map image, $\Phi(M)$, is convex.
\item[Connectedness]  The moment fiber, $\Phi\inv(\alpha)$, is connected
for all $\alpha \in \calT$.
\item[Stability]  As a map to $\Phi(M)$, the moment map is open.
\end{description}
These three properties also hold for the moment map of a local model.
Note that together
the three properties imply that
the moment map preimage of every convex set
is connected.\labell{page:convexity}
Moreover, by  convexity, stability, and the local normal form theorem, 
$\Delta := \Phi(M)$ is a convex polyhedral subset of $\calT$
whose faces have rational slopes.

For the compact case,
see \cite{At}, \cite{GS}, and \cite[Theorem 6.5]{Sj};
also see \cite{LT}.
For convexity and connectedness in the case of proper moment maps
to open convex sets, see \cite{LMTW}. Stability
then follows from the local normal form theorem and stability
for local models; see \cite[Theorem 5.4 and Example 5.5]{Sj}.
Also see~\cite[section 7]{BK}.

\begin{Remark} \labell{rmk:short}

In the situation described above, the set of $\alpha$ in $\Delta$ 
such that the reduced space $\Phi\inv(\alpha)/T$ is a single point
is a union of closed faces of $\Delta$.
To see this,  fix a point $x \in \Delta$ and
let $F_x$ be the smallest (closed) face containing $x$.
The preimage $M_{F_x} = \Phi\inv(F_x)$ is a symplectic manifold with a Hamiltonian
$T$ action. (This follows from the local normal form theorem,
which we will henceforth use without comment.)    Moreover, since $F_x$ is convex,
$M_{F_x}$ is connected.
By stability, the subgroup that acts trivally on $M_{F_x}$ has Lie algebra
$(TF_x)^\circ = \{  \xi \in \ft \mid 
   \langle (y - z) , \xi \rangle = 0 \mbox{ for all } y,z \in F_x\}$.
Moreover, if we assume that $\Phi\inv(x)$ is a single orbit, then stability implies
that $\dim M_{F_x} = 2 \dim F_x$.
Because the moment map level sets of $M$ are connected,
this implies that the level sets over $F_x$ are single orbits, as required.
\end{Remark}

\subsection*{Some consequences} \ 

In order to apply these theorems to complexity one spaces, 
we now analyze complexity one models.

\begin{Lemma} \labell{lemma:Y}
Let $Y = T \times_H \C^{h+1} \times \h^0$
be a complexity one model with moment map
$ \Phi_Y ([t,z,\nu])  = \alpha + \Phi_H(z) + \nu $.
\begin{itemize}
\item
If the moment map $\Phi_Y$ is proper, then the level set
$\Phi_Y\inv(\alpha)$ consists of a single orbit.
\item
If the moment map $\Phi_Y$ is not proper,
then there exists a homeomorphism 
\begin{equation} \labell{homeo}
Y/T \to \left( \image \Phi_Y \right) \times \C
\end{equation}
whose first component is induced from the moment map 
and whose second component takes the set of exceptional
orbits to zero.
\end{itemize}
Moreover, the map~\eqref{homeo} carries the symplectic orientation
of the smooth part of each reduced space to the complex orientation of $\C$.
\end{Lemma}

\begin{proof}
The first assertion
follows from the formula for $\Phi_Y$
and the fact that $\Phi_H$ is quadratic, hence homogeneous.
For the second assertion, see \cite[Lemma~6.2]{locun},
and see \cite[Definition~8.2]{locun} and the sentence that follows it.
\end{proof}

\begin{Corollary}[Short/tall dichotomy]   \labell{dichotomy}
Let $(M,\omega,\Phi,\calT)$ be a complexity one space with moment
image $\Delta = \image \Phi$.  A nonempty reduced space $\Phi\inv(\alpha)/T$ 
is either a connected two dimensional oriented
topological manifold or a single point.
The set of $\alpha$ where the latter occurs is a union of closed faces 
of $\Delta$.
\end{Corollary}

\begin{proof}
By Lemma~\ref{lemma:Y} and the local normal form theorem 
there exists an open set $U \subset M/T$ such that,
for each $\alpha \in \Delta$, 
the intersection $\Phi^{-1}(\alpha)/T \cap U$ is a two dimensional oriented 
topological manifold, and its complement in $\Phi^{-1}(\alpha)/T$
is discrete.
Hence, by the connectedness of the level sets, every non-empty reduced space 
either consists of a single point or is a connected two dimensional oriented 
topological manifold.
The last claim follows from Remark~\ref{rmk:short}.
\end{proof}

By Corollary~\ref{dichotomy}, we can use the following fact 
to understand $\Delta_\tall$.

\begin{Corollary} \labell{tall models}
In a tall complexity one Hamiltonian $T$-manifold, the corresponding
local models are tall.  
\end{Corollary}

\begin{proof}
By Lemma~\ref{lemma:Y},
a complexity one model $Y$ is tall
exactly if there exists a neighbourhood of $\{[t,0,0]\}$ in $Y$
that is tall.
Hence, the claim follows immediately from the local normal form theorem.
\end{proof}

\begin{Corollary} \labell{discrete}
Let $(S,\pi)$ be a tall skeleton over an open subset $\calT$ of $\t^*$.
Then $\pi\inv(\alpha)$ is finite for every $\alpha \in \calT$.
\end{Corollary}

\begin{proof}
Let $s$ be a point in $S$ and let $Y = T \times_H \C^{h+1} \times \h^0$
be the corresponding complexity one model, 
which is tall by Corollary~\ref{tall models}.
By the local normal form theorem and Lemma~\ref{lemma:Y} 
there exists a neighbourhood of $s$ in $S$
whose intersection with $\pi\inv(\alpha)$
consists of the single element set $\{ s \}$.
The result then follows from the properness of $\pi$.
\end{proof}

\section*{Part I: Reconstruction}

{

\section{Topology of complexity one quotients}
\labell{sec:topology}

In this section, we prove two results about the topology
of complexity one quotients which we will need in order to prove
the main propositions in Sections~\ref{sec:lifting} 
and~\ref{sec:gluing}.  For future reference, whenever possible
we will allow complexity one spaces that are not tall.

For a topological space $X$ and a presheaf $\cS$ of abelian groups on $X$,
we let $\check{H}^i(X,\cS)$ denote the \v{C}ech cohomology of $\cS$.
If $X$ is paracompact\footnote{We adopt the convention that, by definition,
every paracompact space is Hausdorff.}, this agrees with the 
\v{C}ech cohomology 
$\check{H}^i(X,\cS^+)$ of the sheafification $\cS^+$ of $\cS$
and with the sheaf cohomology $H^i(X,\cS^+)$ of $\cS^+$ that is defined
through derived functors. 
Voit Th\'eor\`eme 5.10.1 et le Corollaire
de Th\'eor\`eme 5.10.2 de~\cite[chapitre II]{godement}.

Consider a continuous map of topological spaces, $\Phibar \colon Q \to B$. 
Given an abelian group $A$ and a non-negative integer $i$,
define a presheaf $\fH_A^i$ on $B$ by
$$ \fH_A^i(U) = \check{H}^i \big( \Phibar\inv(U);A) \quad 
 \mbox{for each open set } U \subset B. $$
Note that $\fH_A^i(\emptyset) = \{ 0 \} $.
This presheaf is the push-forward  by
$\Phibar \colon Q \to B$ of the presheaf on $Q$
that associates to each open set $W \subset Q$
the group $\check{H}^i(W;A)$.
In general, neither presheaf is a sheaf.  

\begin{Proposition} \labell{prop zero coh}
Let $Q$ be a topological space, $\calT$  be an open subset of~$\ft^*$, and 
$ \Phibar \colon Q \to \calT $
be a continuous map such that 
$\Delta = \image \Phibar$ 
is convex.
Assume that for every point in $\calT$ there exists a convex neighbourhood $U$
in $\calT$, a complexity one space $(M_U,\omega_U,\Phi_U,U)$, 
and a homeomorphism from $\Phibar\inv(U)$ to $M_U/T$ 
that carries $\Phibar|_{\Phibar\inv(U)}$ 
to the map $\Phibar_U \colon M_U/T \to U$ induced  by $\Phi_U$.
Then for any abelian group $A$,
$$
   \check{H}^i \big( \calT,\fH_A^0 \big) = 
   \check{H}^i \big( \calT,\fH_A^1 \big) = 0 
   \quad \text{ for all }   i > 0 . 
$$
Moreover, if at least one of the spaces $M_U$ is not tall, then 
$$ \check{H}^0 \big( \calT,\fH_A^2 \big) = 0.$$
\end{Proposition}

\begin{proof}
We first show that $Q$ is paracompact. Let $\fW$ be an arbitrary open covering
of $Q$.  There exists a locally finite covering $\nu$ of $\calT$
by open balls such that every $B \in \nu$, the preimage
$\Phibar\inv(\ol{B})$ of the closure $\ol{B}$ of $B$ is compact.
For each $B \in \nu$, let $\fW_B \subset \fW$ be a finite subset
that covers $\Phibar\inv(\ol{B})$; then 
$$ \bigcup\limits_{B \in \nu} 
   \left\{ W \cap f\inv(B) \ | \ W \in \fW_B \right\} $$
is a locally finite open refinement of $\fW$ that covers $Q$.

The map $\Phibar \colon Q \to \Delta \subset \calT$ induces 
presheaves $\fH_A^j$ 
on $\Delta$ and $\calT$.
Moreover, since $\fH_A^j(U) = \fH_A^j(U \cap \Delta)$ for all open $U \subset \calT$,
we have $$\check{H}^i \big(\calT,\fH_A^j \big) = 
\check{H}^i \big(\Delta,\fH_A^j \big) \quad \mbox{for all } i \mbox{ and } j.$$

Let $(\fH_A^j)^+$ denote the sheafification of the presheaf 
$\fH_A^j$ on $\Delta$.  
Because the \v{C}ech cohomology of a presheaf on a paracompact
space is equal to that of its sheafification,
it is enough to prove that
$$ \check{H}^i \big( \Delta,(\fH_A^0 )^+ \big) 
 = \check{H}^i \big( \Delta,(\fH_A^1 )^+ \big) = 0
 \quad \text{ for all } i > 0  $$
and that, if at least one of the spaces $M_U$ is not tall, then 
$$\check{H}^0 \big( \Delta,(\fH_A^2)^+ \big) = 0.$$

Assume first
that all of the complexity one spaces $M_U$ are tall.
By Proposition \ref{trivialize M mod T},  this implies that
for every  point in $\calT$ there exists a convex neighbourhood $U$
in $\calT$,  a surface $\Sigma$,
and a function $f \colon \Phibar\inv(U) \to \Sigma$, such that 
$$ \big( \Phibar,f \big) \colon \Phibar\inv(U) 
   \to (\Delta \cap U) \times \Sigma $$
is a homeomorphism.
Hence, $(\fH_A^0)^+$ is a constant sheaf and $(\fH_A^j)^+$ is 
a locally constant sheaf for all $j > 0$.
Since $\Delta$ is convex, it is contractible; thus
$\check{H}^i \big( \Delta,(\fH_A^j)^+ \big) = \{ 0 \}$ 
for all $j$ and all $i > 0$.

Next, assume that at least one of the complexity
one spaces $M_U$ is not tall.
Let $\Delta_\tall$ denote the set of $\alpha \in \Delta$
such that $\Phibar\inv(\alpha)$ is a connected
two dimensional oriented topological manifold;
let $\Delta_\short = \Delta \ssminus \Delta_\tall$.
Corollary~\ref{dichotomy} implies that $\Delta_\tall$ is open in $\Delta$ and 
$\Phibar\inv(\alpha)$ is a single point for all $\alpha \in \Delta_\short$.

By assumption, for every point in $\calT$ there exists a convex 
neighbourhood $U$, a complexity one space $(M_U,\omega_U,\Phi_U,U)$, 
and a homeomorphism from $\Phibar\inv(U)$ to $M_U/T$ 
that carries $\Phibar|_{\Phibar\inv(U)}$ 
to the map $\Phibar_U \colon M_U/T \to U$ induced by $\Phi_U$.
In fact, the convexity package 
implies that the preimage $\Phi_U\inv(V) $ is connected for any convex subset 
$V \subset U$; see page~\pageref{page:convexity}.
Hence, the neighbourhood $U$ can be chosen to be arbitrarily small.

In particular,
every $\alpha \in \Delta$ has arbitrarily small neighbourhoods
whose pre-images in $Q$ are connected. 
Hence, $(\fH_A^0)^+$ is a constant sheaf.
Since $\Delta$ is convex, this implies that
$\check{H}^i\big(\Delta,(\fH_A^0)^+\big) = 0$ for all $i > 0$.

The following result is proved in~\cite[Lemma~5.7]{locun}:
\begin{equation} \labell{5.7}
\begin{minipage}{4in}
Let $(M,\omega,\Phi,U)$ be a complexity one space.
Suppose that $\Phi\inv(\alpha)$ consists of a single orbit.
Then every neighbourhood of $\alpha$ contains a smaller neighbourhood $V$ 
such that the quotient $\Phi\inv(V)/T$ is contractible.
Moreover, every regular non-empty symplectic quotient $\Phi\inv(y)/T$
in $\Phi\inv(V)/T$ is homeomorphic to a $2$-sphere.
\end{minipage}
\end{equation}

By Proposition~\ref{trivialize M mod T},
the genus
of the reduced space is locally constant on $\Delta_\tall$.
Hence, since regular values are dense,
\eqref{5.7} implies that this genus is zero for all two dimensional
reduced spaces over a neighbourhood of $\Delta_\short$.
Since $\Delta$ is connected and $\Delta_\short$ is not empty,
this implies that every two dimensional reduced space has genus zero.
Hence, by Proposition~\ref{trivialize M mod T}
and~\eqref{5.7},  $(\fH_A^1)^+$ is the zero sheaf.
Therefore, $\check{H}^i\big(\Delta,(\fH_A^1)^+\big) = 0$ for all $i > 0$.

Finally, consider a global section
$\gamma \in \check{H}^0 \big(\Delta,(\fH_A^2)^+ \big)$.
By~\eqref{5.7}, the support of $\gamma$
is a subset of $\Delta_\tall$.
Therefore,
since the restriction of $(\fH_A^2)^+$ to $\Delta_\tall$ is
a locally constant sheaf by Proposition~\ref{trivialize M mod T},
the support of $\gamma$ is  an open and closed subset of $\Delta_\tall$.
Since $\Delta$ is connected and $\Delta_\short$
is non-empty, this implies that $\gamma = 0$.
Thus, $\check{H}^0\big(\Delta,(\fH_A^2)^+\big) = 0$.
\end{proof}

\begin{Proposition} \labell{prop121}
Let $(M,\omega,\Phi,\calT)$ be a complexity one space.
The restriction map $H^2(M/T;\Z) \to H^2 \big(\Phi\inv(y)/T;\Z \big)$
is one-to-one
for each $y \in \image \Phi$.
\end{Proposition}

\begin{proof}
If the complexity one space is tall, this proposition 
is an immediate consequence of Proposition \ref{trivialize M mod T}.
So assume that it is not tall.
Let $\Phibar \colon M/T \to \calT$ be the map
induced by $\Phi$.  
Then there is the Leray spectral sequence 
converging to $H^*(M/T;\Z)$
with
$$ E_2^{i,j} = \check{H}^i \big(\calT,\fH_\Z^j \big); $$
see \cite[chap.~II, Thm.~4.17.1]{godement}.
By Proposition \ref{prop zero coh}, $E_2^{i,j} = 0$ 
for all $i$ and $j$ such that $i+j = 2$.
Consequently, $H^2(M/T;\Z) = 0$.
\end{proof}

}

\section{Lifting from the quotient}
\labell{sec:lifting}

An important step in gluing together local pieces of complexity
one spaces is to glue them together as $T$-manifolds.  To carry this
out, which we will do in this section,
we need a notion of diffeomorphisms of quotient spaces.

Let a compact torus $T$ act on a manifold $N$.
The quotient $N/T$ can be given
a natural differential structure,
consisting of the sheaf of real-valued functions
whose pullbacks to $N$ are smooth.\footnote{
This notion of a differential structure on quotient spaces
was used by Schwarz \cite{schwarz:IHES}.
An axiomatization of ``differential structure" appeared
in \cite{sikorski}.}
We say that a map $h \colon N/T \to N'/T$ is \textbf{smooth}
if it pulls back smooth
functions to smooth functions; it is a \textbf{diffeomorphism}
if it is smooth and has a smooth inverse.
If $N$ and $N'$ are oriented, the choice of an orientation on $T$
determines orientations on the smooth part of $N/T$ and $N'/T$.
Whether or not a diffeomorphism $f\colon N/T \to N'/T$ preserves orientation
is independent of this choice.

We now recall several definitions from~\cite{locun}.

\begin{Definition}
Let a torus $T$ act on oriented manifolds $M$ and $M'$ with
$T$-invariant maps $\Phi \colon M \to \t^*$ and $\Phi' \colon M' \to \t^*$.
A \textbf{$\bs \PhiT$-diffeomorphism} from $(M,\Phi)$ to  $(M',\Phi')$
is an orientation preserving equivariant diffeomorphism
$f \colon M \to M'$ that satisfies $\Phi' \circ f = \Phi$.
\end{Definition}

\begin{Definition} \labell{Phi diffeo}
Let $(M,\omega,\Phi,\calT)$ and $(M',\omega',\Phi',\calT)$
be complexity one Hamiltonian $T$-manifolds.
A $\mathbf{\Phi}$\textbf{-diffeomorphism} from $M/T$ to $M'/T$
is an orientation preserving diffeomorphism  $f \colon M/T \to M'/T$
such that ${\Phibar}' \circ f = \Phibar$, and such that
 $f$ and $f\inv$ lift to  $\PhiT$-diffeomorphisms
in a neighbourhood of each exceptional orbit.
Here, $\Phibar$ and ${\Phibar}'$ are induced by the moment maps.
\end{Definition}

We now state the main result of this section.

\begin{Proposition} \labell{glue lifts}
Let $\calT \subset \t^*$ be an open subset,
$\Delta \subset \calT$ a convex subset,
and $\rho \colon \Delta \to \R_{> 0}$ a function.
Let $\fU$ be a cover of $\calT$ by convex open sets.
For each $U \in \fU$, let $(M_U,\omega_U,\Phi_U)$ be
a complexity one space over $U$ with $\image \Phi_U = U \cap \Delta$
and \DH function $\rho|_U$.
For each $U$ and $V$ in $\fU$, let
$$f_{UV} \colon M_V|_{U \cap V} / T \to M_U|_{U \cap V} / T$$
be a $\Phi$-diffeomorphism,
such that $f_{UV} \circ f_{VW} = f_{UW}$
on $M_W|_{U \cap V \cap W} / T$ for all $U,V,W \in \fU$.
Then, after possibly passing to a refinement of the cover,
there exist $\PhiT$-diffeomorphisms 
$g_{UV} \colon M_V|_{U \cap V} \to M_U|_{U \cap V}$
that lift $f_{UV}$ and such that $g_{UV} \circ g_{VW} = g_{UW}$
on $M_W|_{U \cap V \cap W}$ for all $U,V,W \in \fU$.
\end{Proposition}

Under the assumptions of Proposition \ref{glue lifts},
let $Q$ denote the topological space obtained
from the disjoint union $\bigsqcup_{U \in \fU} M_U / T$
by identifying $x$ with  $f_{UV} (x)$  for all 
$U$ and $V$ in $\fU$
and all $x \in M_V|_{U\cap V}/T$.
 Let
$$ \Phibar \colon Q \to \calT $$
denote the map induced by the moment maps. 
As in the proof of Proposition~\ref{prop zero coh}, $Q$ is paracompact.

We define a differential structure on $Q$ by declaring a real-valued
function to be smooth if it lifts to a smooth function on each $M_U$;
notice that this is well defined. 
We can use smooth partitions of unity on the spaces $M_U$ and $\calT$
to construct smooth partitions of unity on $Q$.

For any abelian Lie group $A$, 
let $A^\infty$ denote the sheaf of smooth functions to $A$.
Let $\calH^i_{A^\infty}$ denote the presheaf on $\calT$ 
which associates the group 
$\check{H}^i(\Phibar\inv(U);A^\infty)$ to each open set $U \subset \calT$.
We will need the following lemma:

\begin{Lemma} \labell{H2 is zero}
In the above situation,
$$ \check{H}^2( \calT , \fH^0_{T^\infty}) = 0 .$$
\end{Lemma}

\begin{proof}
Every short exact sequence
of sheaves on $Q$ gives rise to a long exact sequence
in \v{C}ech cohomology.  
Therefore,
the short exact sequence of  sheaves on $Q$,
$$ 0 \to \ell \to \t^\infty \to T^\infty \to 1,$$
gives rise to a long exact sequence of presheaves on $\calT$
\begin{equation} \labell{exact of calH}
 0 \to \calH^0_\ell \to \calH^0_{\t^\infty} 
\to \calH^0_{T^\infty} 
\to \calH^1_\ell \to  \calH^1_{\t^\infty} \to \cdots .
\end{equation}

Because the sheaf $\t^\infty$ is fine, 
$\check{H}^1(W,\t^\infty) = 0$ 
for all open sets $W \subset Q$. 
Hence, $\calH^1_{\t^\infty}$ is the zero presheaf.
Thus \eqref{exact of calH}
breaks up into two short exact sequences of presheaves,
$$
   0 \to \calH^0_\ell \to \calH^0_{\t^\infty} \to \kappa \to  0
\quad \text{ and } \quad
   0 \to \kappa \to \calH^0_{T^\infty} 
     \to \calH^1_\ell \to 0 ,
$$
where $\kappa$ denotes the kernel of the homomorphism
$\calH^0_{T^\infty} \to \calH^1_\ell$.
From these short exact sequences we get long exact sequences
\begin{equation} \labell{long1}
\cdots \to \check{H}^i(\calT, \calH^0_{\t^\infty}) \to \check{H}^i(\calT,\kappa)
\to  \check{H}^{i+1}(\calT,\calH^0_\ell)  
\to \check{H}^{i+1}(\calT,\calH^0_{\t^\infty})
\to \cdots 
\end{equation}
and
\begin{equation} \labell{long2}
 \cdots \to  \check{H}^i(\calT,\kappa) \to \check{H}^i(\calT,\calH^0_{T^\infty})
 \to \check{H}^i(\calT,\calH^1_\ell) \to \cdots.
\end{equation}

Because $\calH^0_{\t^\infty}(U) = \t^\infty(\Phibar\inv(U))$ 
for all $U \subset \calT$, the sheaf $\calH^0_{\t^\infty}$
is a fine sheaf, and so 
$$
\check{H}^i(\calT,\calH^0_{\t^\infty}) = 0 
\quad \text{ for all } \quad i > 0.
$$
Hence, \eqref{long1} implies that
$
\check{H}^{i}(\calT,\kappa) =
\check{H}^{i+1}(\calT,\calH^0_\ell) $ 
for all
$i > 0$. 
Thus, \eqref{long2} becomes 
$$
 \cdots \to  \check{H}^{i+1}(\calT,\fH^0_\ell) \to \check{H}^i(\calT,\calH^0_{T^\infty})
 \to \check{H}^i(\calT,\calH^1_\ell) \to \cdots
$$
for all $i > 0$.
The claim now follows immediately from Proposition~\ref{prop zero coh}.
\end{proof}

The proof of Proposition \ref{glue lifts} will use
the following result.

\begin{Lemma} \labell{global lift}
Let $(M,\omega,\Phi,U)$ and $(M',\omega',\Phi',U)$ be complexity one 
spaces that have the same \DH function.  Then
every $\Phi$-diffeomorphism $f \colon M/T \to M'/T$ 
lifts to a $\PhiT$-diffeomorphism from $M$ to $M'$.
\end{Lemma}

\begin{proof}
Lemma 4.10 of \cite{locun} reads as follows: 
\begin{quotation}
Let $Y$ be a local model for a non-exceptional orbit with a moment map
$\Phi_Y \colon Y \to \t^*$.
Let $W$ and $W'$ be invariant open subsets of $Y$.
Let $g \colon W/T \to W'/T$  be a diffeomorphism which preserves
the moment map.
Then $g$ lifts to an equivariant diffeomorphism from $W$ to $W'$.
\end{quotation}
Therefore, by  Definition \ref{Phi diffeo}  and the 
local normal form theorem,
every orbit in $M/T$ has a neighbourhood on which $f$ lifts 
to a $\PhiT$-diffeomorphism.

Condition (3.2) of \cite{locun} reads as follows: 
\begin{equation}\labell{technical121A}\tag{*} 
\begin{array}{l}
\text{The restriction map }
        H^2(M/T;\Z) \to H^2(\Phi\inv(y)/T;\Z) \\
\text{is one-to-one for some regular value $y$ of $\Phi$.}
\end{array}
\end{equation}
Lemma 4.11 of \cite{locun} reads as follows:
\begin{quotation}
Let $(M,\omega,\Phi,U)$ and $(M',\omega',\Phi',U)$ be complexity one spaces
that satisfy Condition \eqref{technical121A} 
and have the same \DH measure.  Then
every homeomorphism from $M/T$ to $M'/T$ that locally lifts
to a $\PhiT$-diffeomorphism also lifts globally
to a $\PhiT$-diffeomorphism from $M$ to~$M'$.
\end{quotation}
The lemma follows from this and Proposition \ref{prop121}.
\end{proof}

We will also need the following result from \cite{HaSa}:

\begin{refTheorem}[\cite{HaSa}]\labell{schwarz}
Let a torus $T$ act on a manifold $M$. Let $h \colon M \to M$ be an
equivariant diffeomorphism that sends each orbit to itself.
Then there exists a smooth invariant function $f\colon M \to T$
such that $h(m) = f(m) \cdot m$ for all $m \in M$.
\end{refTheorem}

\begin{proof}[Proof of Proposition \ref{glue lifts}]
Fix any $U$ and $V$ in $\fU$.
Since $U$ and $V$, and hence $U \cap V$, are convex,
we can apply Lemma \ref{global lift} to the spaces 
$M_U|_{U \cap V}$ and $M_V|_{U \cap V}$.
Thus, 
there exists a $\PhiT$-diffeomorphism
$$ F_{UV} \colon M_V |_{U \cap V} \to M_U |_{U \cap V} $$
that lifts $f_{UV}$.

For every $U, V, W \in \fU$,
by Theorem \ref{schwarz},
$F_{UV} \circ F_{VW} \circ F_{UW}\inv$ is given by acting by
a smooth $T$-invariant function $M_U|_{U \cap V \cap W} \to T$.
This function is the pull-back of a smooth function
\begin{equation} \labell{hUVW}
 h_{UVW} \colon Q|_{U \cap V \cap W} \to T .
\end{equation}
On quadruple intersections, we have
\begin{equation} \labell{h is closed}
 (h_{UVW}) (h_{UVX})\inv (h_{UWX}) (h_{VWX})\inv = 1.
\end{equation}
This is a cocycle condition; hence, the $h_{UVW}$ 
represent a cohomology class in $H^2(\fU,\fH^0_{T^\infty})$.
By Lemma \ref{H2 is zero}, after possibly passing to a refinement
of the cover $\fU$, there exist smooth $T$-invariant functions
$$B_{UV} \colon Q|_{U \cap V} \to T$$
such that
\begin{equation} \labell{h is exact}
B_{UV} \; B_{VW} \; B_{UW}\inv = h_{UVW}
\end{equation}
on triple intersections. 

Then
$$ g_{UV}(x) := 
   \left(B_{UV}(x)\right)\inv \cdot F_{UV}(x) 
$$
are liftings of the $f_{UV}$'s that satisfy 
the required compatibility condition.
\end{proof}


\section{Gluing symplectic forms}
\labell{sec:gluing}

The last ``local to global" step is to modify the symplectic forms
on the local pieces so that they agree on overlaps.

\begin{Proposition} \labell{glue forms given DH}
Let an $n-1$ dimensional torus $T$ act on an oriented $2n$ dimensional 
manifold $M$, 
and let $\Phi \colon M \to \calT$ be an invariant proper map
to an open subset $\calT \subset \t^*$.
Assume that $\Delta = \image \Phi$ is convex. 
Fix a function $\rho \colon \Delta  \to \R_{>0}$
and  an open cover $\fU$ of $\calT$.

Assume that, for all $U \in \fU$,  there exists an invariant symplectic 
form $\omega_U$ on $\Phi\inv(U)$ with moment map $\Phi|_U$ and \DH function 
$\rho|_U$ such that $\omega_U$ is compatible with the given orientation.
Then there exists an invariant symplectic form $\omega'$ on $M$ 
with moment map $\Phi$ and \DH function $\rho$ such that $\omega'$ 
is compatible with the given orientation.
\end{Proposition}

Let a compact Lie group $G$ act on a manifold $M$,
and let $\{\xi_M\}_{\xi \in \g}$  be the vector fields
that generate this action.
A differential form $\beta$ on $M$ is \textbf{basic} if
it is $G$ invariant and horizontal,  that is,
$\iota_{\xi_M} \beta =0 $ for all $\xi \in \g$.
The basic differential forms on $M$ constitute a differential
complex $\Omega^*_{\operatorname{basic}}(M)$  
whose cohomology coincides with the \v{C}ech cohomology
of the topological quotient $M/G$; see \cite{koszul}.

We will need the following technical lemma; 
cf.~\cite[Lemma 3.6]{locun}.

\begin{Lemma} \labell{nondegenerate}
Let an $(n-1)$ dimensional abelian group $T$ act faithfully on
a $2n$ dimensional manifold $M$. 
Let $\Phi \colon M \to \t^*$ be a smooth invariant map.
Let $\omega_0$ and $\omega_1$ be invariant symplectic forms on $M$ 
with moment map  $\Phi$ that induce the same orientation on $M$.
Let $\alpha$ be a basic two-form on $M$
such that $\alpha(\xi,\eta) = 0$ for all $\xi,\eta \in \ker\, d\Phi$.
Let $\lambda_0$ and $\lambda_1$ be non-negative functions on $M$ such that 
$\lambda_0 + \lambda_1 = 1$.  Then
$$\lambda_0 \, \omega_0 +  \lambda_1 \, \omega_1 + \alpha $$
is non-degenerate and induces the same orientation
as $\omega_0$ and $\omega_1$.
\end{Lemma}

\begin{proof}
Let $x \in M$ be a point with stabilizer $H$; let $h$ be
the dimension of $H$.
By the local normal form theorem, a neighbourhood of the orbit of $x$
with the symplectic form $\omega_0$ is equivariantly symplectomorphic
to a neighbourhood of the orbit $\{ [t,0,0] \}$ in the model
$T \times_H \C^{h+1} \times \h^0$. 
The tangent space at $x$ splits as
$\t / \h \oplus \h^0 \oplus \C^{h+1}$, where $\t/\h$ is the tangent space
to the orbit.
By the definition of the moment map, the forms ${\omega_0}|_x$
and ${\omega_1}|_x$ are given by block matrices of the form
$$ \left( \begin{array}{ccc} 0 & I & 0 \\ -I & * & * \\
        0 & * & \omegabar_0 \end{array} \right)
\qquad \text{ and } \qquad
   \left( \begin{array}{ccc} 0 & I & 0 \\ -I & * & * \\
        0 & * & \omegabar_1 \end{array} \right)
$$
where $I$ is the natural pairing between the vector space $\t/\h$
and its dual, $\h^0$, and where $\omegabar_0$ and $\omegabar_1$ are
linear symplectic forms on $\C^{h+1}$ with the same moment map and the
same orientation. 
By our assumptions, $\alpha|_x$ is given by a block matrix of the form
$$ \left( \begin{array}{ccc} 0 & 0 & 0 \\ 0 & * & * \\
        0 & * & 0 \end{array} \right) .
$$
Hence, 
$(\lambda_0 \, \omega_0 + \lambda_1 \, \omega_1 + \alpha)|_x$
is given by a block matrix of the form
$$ \left( \begin{array}{ccc} 0 & I & 0 \\ - I & * & * \\
          0 & * & \omegabar \end{array} \right)
$$
where
$$ \omegabar = \lambda_0(x) \omegabar_0 + \lambda_1(x) \omegabar_1.$$ 
It suffices to show that $\omegabar$ is non-degenerate
 and induces the same orientation
as $\omegabar_0$ and $\omegabar_1$.

\noindent \textbf{Case 1.}
Suppose that the stabilizer of $x$ is trivial.
Then $\omegabar_0$ and $\omegabar_1$ are non-zero two-forms on $\C$
that induce the same orientation, 
and so $\omegabar$ is non-degenerate and induces the same orientation.

\noindent \textbf{Case 2.}
Suppose that the stabilizer of $x$ is non-trivial.
Viewing $\omegabar$ as a translation invariant differential two-form 
on $\C^{h+1}$, it is enough to find some $v \in \C^{h+1}$
such that $\omegabar|_v$ is non-degenerate and induces the same orientation 
as $\omegabar_0|_v$ and $\omegabar_1|_v$.
We choose $v \in \C^{h+1}$ whose stabilizer is trivial and apply Case 1 
to the $H$ action on $\C^{h+1}$.
\end{proof}

\begin{proof}[Proof of Proposition \ref{glue forms given DH}]
Given $j \in \N$, 
define a sheaf   $\Otb^j$  on $\calT$
by
$$\Otb^j(U) = \Omega^j_{\operatorname{basic}}(\Phi\inv(U))
\quad \mbox{for all open }U \subset \calT.$$
Consider the double complex
$$ K^{i,j} = \check{C}^i(\fU,\Otb^j).$$
Let $d \colon  K^{i,j} \to K^{i,j+1}$ 
denote the de-Rham differential, 
and let $\delta \colon K^{i,j} \to K^{i+1,j}$
denote the \v{C}ech differential.

To prove the proposition, it will be enough to find 
$\beta \in C^1(\fU,\Otb^1)$ such that $\delta \beta = 0$
and $d \beta_{VW} = \omega_V - \omega_W$ for all $V$ and $W$ in $\fU$.
To see this, let $\{\lambda_U\}_{U \in \fU}$ 
be the pull back to $M$
of a smooth partition of unity on
$\calT$ subordinate to $\fU$.
Define
$$\omega'_V := \sum_{U \in \fU} \lambda_U\,  \omega_U + 
\sum_{U \in \fU} d \lambda_U \wedge \beta_{UV}  \  \in \Omega^2(\Phi\inv(V))
\quad \mbox{for all } V \in \fU.$$
Since $\delta \beta   = 0$ and $\sum_{U \in \fU} \lambda_U = 1$,
$$\omega_V' - \omega_W' 
= \sum_{U \in \fU} d \lambda_U \wedge (\beta_{UV} - \beta_{UW}) 
= d \Big( \sum_{U \in \fU} \lambda_U \Big) \wedge \beta_{WV} = 0.$$
Therefore, the $\omega_V'$ glue together to give a global form 
$\omega' \in \Omega^2(M)$.
Since each $\omega_U$ is an invariant symplectic form, and since
$d \lambda_U(\xi) = 0$ for all $U \in \fU$ and all $\xi \in \ker \  d \Phi$,
by repeated application of Lemma~\ref{nondegenerate}
$\omega'$ is non-degenerate and is compatible with the given orientation.
Moreover,
\begin{multline*}
\omega_V + \sum_{U \in \fU} d(\lambda_U \, \beta_{UV})   
 = \omega_V + \sum_{U \in \fU} \big( d \lambda_U \wedge \beta_{UV} +
 \lambda_U \, (\omega_U - \omega_V) \big) \\
=   \omega_V - \sum_{U \in \fU} \lambda_U \, \omega_V + \omega'_V 
 = \omega'_V.
\end{multline*}
Thus, each $\omega_V$ and $\omega'_V$ differ by the exterior derivative
of a basic one-form.
This implies that $\omega'$ is closed, and so it is a symplectic
form compatible with the given orientation.
It also implies that $\omega'$ is invariant, has the same moment map $\Phi$
as $\omega_V$, and has the same \DH function $\rho$ as $\omega_V$.

As a first step towards finding the required cochain,
we will show
that we may assume that
there exists $\beta \in C^1(\fU,\Otb^1)$ 
such that $d \beta_{VW} = \omega_V - \omega_W$ for all $V$ and $W$ in $\fU$.
After possibly passing to a refinement, we
may assume that every $U \in \fU$ is convex.
As we mentioned earlier,  Condition (3.2) of \cite{locun} reads
as follows:
\begin{equation}\labell{technical121B}\tag{*} 
\begin{array}{l}
\text{The restriction map }
        H^2(M/T;\Z) \to H^2(\Phi\inv(y)/T;\Z) \\
\text{is one-to-one for some regular value $y$ of $\Phi$.}
\end{array}
\end{equation}
Moreover, Lemma 3.5 of \cite{locun} reads
as follows:
\begin{quotation}
Let $(M,\omega,\Phi,U)$ and $(M',\omega',\Phi',U)$ be
complexity one spaces
that satisfy Condition \eqref{technical121B}
and have the same \DH measure.
Then for every $\PhiT$-diffeomorphism $g \colon M \to M'$
there exists  a basic one-form $\beta$ on $M$
such that $\d \beta = g^* \omega' - \omega$.
\end{quotation}
Hence, the claim follows immediately from Proposition \ref{prop121}.

Next, we will show that we may assume that there exists 
$\gamma \in \check{C}^2(\fU,\Otb^0)$ 
such that $\delta \gamma = 0$ and $\delta \beta = d \gamma$.
For all $j \in \N$, define a presheaf $\fH^j_\R$ 
on $\calT$  
by $\fH_\R^j(U) = \check{H}^j(\Phi\inv(U)/T;\R)$  
for all open $U \subset \calT$.
Recall that 
the \v{C}ech cohomology of $\Phi\inv(U)/T$
coincides with the cohomology of 
$\left( \Omega^*_{\operatorname{basic}}(\Phi\inv(U)), d \right)$.
Since $\delta^2 \beta = 0$ 
and $d \delta \beta = 0$,
the cochain $\delta \beta$
represents  a cohomology class in $\check{H}^2(\fU;\fH^1_\R)$.
By Proposition \ref{prop zero coh}, $\check{H}^2(\calT,\fH^1_\R) = 0$.
Hence,  after passing to a refinement,
there exists $\beta' \in \check{C}^1(\fU,\Otb^1)$
such that $d\beta'=0$ and such that $\delta \beta$ and $\delta \beta'$
agree as elements of  $\check{C}^2(\fU,\fH^1_\R)$, i.e., 
there exists $\gamma \in \check{C}^2(\fU,\Otb^0)$
such that $\delta \beta - \delta \beta' = d  \gamma$.
By replacing $\beta$ by $\beta - \beta'$, 
we may assume that $\delta \beta = d \gamma$, as required.
Since $\check{H}^3(\calT,\fH^0_\R) = 0$ by Proposition \ref{prop zero coh}, 
we may assume that $\delta \gamma = 0$ by a similar argument.

Finally, we will use the fact that $\Otb$ is a fine sheaf
to show that 
 we may assume that $\delta \beta = 0$, as required.
Define $\eta \in \check{C}^1(\fU,\Otb^0)$ by 
$$\eta_{VW} =  \sum_{U \in \fU} \lambda_U \, \gamma_{UVW} 
   \quad \mbox{for all} \quad V,W \in \fU.$$
Since $\delta \gamma = 0$, \ $\delta \eta = \gamma$, and so
$\delta d \eta = d \gamma = \delta \beta$.
Hence,  we may replace $\beta$ by $\beta - d \eta$.
\end{proof}

\section{Reconstruction}
\labell{sec:proof reconstruction}

By breaking a space into the moment map
preimages of small open subsets of $\ft^*$, 
and then gluing them back together, 
we obtain a special case of Theorem \ref{data from manifold}.
This theorem is easier to prove than our other existence theorems,
in that it does not require the ``local existence" results proved in 
Sections~\ref{sec:compatible skeleton}--\ref{sec:DH compatible},
i.e., it does not 
require us to determine which spaces can occur as preimages 
of small open subsets of~$\t^*$.

\begin{refTheorem} \labell{reconstruction}
Let $(M,\omega,\Phi,\calT)$ be a tall complexity one space
of genus~$g$.  Let $\Sigma$ be a closed oriented surface of genus $g$,
and let  $f \colon M_\exc \to \Sigma$ be any painting. 
Then there exists a tall complexity one space
$(M',\omega',\Phi',\calT)$ with the same moment image and  \DH  function as $M$
whose painting is equivalent to $f$.
\end{refTheorem}

Since every tall complexity one space
$(M,\omega,\Phi,\calT)$ has a  convex moment image 
$\Delta = \Phi(M)$, 
a positive \DH function $\rho$,
and a skeleton $S = M_\exc$, 
Theorem~\ref{reconstruction} is simply the special
case of Proposition~\ref{combined} below with $\fU = \{ \calT \}$.
(Proposition~\ref{combined} is also a key ingredient in the
proofs of Theorems \ref{data from manifold}, \ref{thm:existence}, 
and \ref{existence}; see  \S\ref{sec:proof}.)

\begin{Proposition} \labell{combined}
Let $\calT$ be a convex open subset of $\ft^*$,
$\Delta \subset \calT$  a convex subset,
and $\rho \colon \Delta \to \R_{> 0}$  a  positive function.
Let $(S,\pi)$ be a skeleton over $\calT$, 
\ $\Sigma$ a closed oriented surface of genus $g$, 
and  $f \colon S \to \Sigma$  a painting.
Finally, let $\fU$ be a cover of $\calT$ by convex open sets.

Suppose that for each  $U \in \fU$
there exists 
a complexity one space $(M_U,\omega_U,\Phi_U)$ of genus $g$ over $U$ 
with moment image $\Delta \cap U$ and \DH function $\rho|_{\Delta \cap U}$ 
whose skeleton is isomorphic to $S \cap \pi\inv(U)$.
Then there exists a complexity one space $(M,\omega,\Phi,\calT)$
with moment image  $\Delta$ and \DH function $\rho$
whose painting is equivalent to $f$.
\end{Proposition}

\begin{proof}
By Proposition 20.1 from the elephant \cite{globun} 
(see Proposition~\ref{technical} below),
after (possibly) passing to a refinement of $\fU$, there
exists $\Phi$-diffeomorphisms
$h_{V U} \colon M_U/T|_{U \cap V} \to  M_V/T|_{U \cap V}$
such that $h_{W V} \circ h_{V U} = h_{W U}$ on triple intersections
and the following property holds.

\begin{quotation}
If $(M,\omega,\Phi,\calT)$ is a complexity one space
such that for every $U \in \fU$ there
exists a $\PhiT$-diffeomorphism $\lambda_U \colon M|_U \to M_U$
so that $h_{V U}$ is the map induced by
the composition $\lambda_V \circ (\lambda_U)\inv$,
then the painting associated to $M$ is equivalent to $f$.
\end{quotation}

By Proposition \ref{glue lifts}, after passing
to a refinement of $\fU$, there exist $\PhiT$-diffeomorphisms
$g_{VU} \colon M_U|_{U \cap V} \to M_V|_{U \cap V}$ 
that lift $h_{VU}$ and such that
$g_{WV} \circ g_{VU} = g_{WU}$
on every triple intersection.

We use these $\PhiT$-diffeomorphisms to glue together the manifolds~$M_U$.
This gives an oriented $2n$-dimensional manifold $M$ with a $T$ action 
and a $T$-invariant proper map $\Phi \colon M \to \calT$. 
Moreover, there exists a  $\PhiT$-diffeomorphism
$\lambda_U \colon M|_U \to M_U$ for each $U \in \fU$
such that $g_{VU} = \lambda_V \circ (\lambda_U)^{-1}$
on each double intersection.

For each $U \in \fU$, the pullback $\omega_U' := \lambda_U^* \omega_U$  
is a $T$-invariant symplectic form on $\Phi\inv(U)$
with moment map $\Phi|_U$ and \DH function $\rho|_U$
and such that $\omega_U'$  
is compatible with the  given orientation.
Therefore, by Proposition \ref{glue forms given DH},
there exists a $T$-invariant symplectic form $\omega$ on $M$
with moment map $\Phi$ and \DH function $\rho$ 
and such that $\omega$ is compatible with the given orientation.
Finally, by the property above, the painting associated 
to $(M,\omega,\Phi,\calT)$ is equivalent to $f$.
\end{proof}

For the reader's convenience, we now reformulate
Proposition 20.1 from \cite{globun}:

\begin{Proposition} \labell{technical}
Let $\calT$ be an open subset of $\ft^*$
and $\Delta \subset \calT$ a convex closed subset.
Let $f \colon S \to \Sigma$ be a painting, where $\Sigma$
is a closed oriented surface of genus $g$
and $(S,\pi)$ is a skeleton over $\calT$.
Let $\fU$ be a cover of $\calT$ by convex open sets.
For each $U \in \fU$, let $(M_U,\omega_U,\Phi_U,U)$ be a 
tall complexity one space of genus $g$ over $U$,
so that $\image \Phi_U = U \cap \Delta$
and so that the set of exceptional orbits $(M_U)_\exc$
is isomorphic to the restriction $S|_U := S \cap \pi\inv(U)$.

Then, after possibly refining the open cover,
one can associate to each $U$ and $V$ in $\fU$ a 
$\Phi$-diffeomorphism
$h_{V U} \colon M_U/T|_{U \cap V} \to  M_V/T|_{U \cap V}$
such that $h_{W V} \circ h_{V U} = h_{W U}$,
and such that the following holds.

If $(M,\omega,\Phi,\calT)$ is a 
tall complexity one space
such that for every $U \in {\mathcal U}$ there
exists a $\PhiT$-diffeomorphism $\lambda_U \colon M|_U \to M_U$
so that $h_{V U}$ is the map induced by
the composition $\lambda_V \circ (\lambda_U)\inv$,
then the painting associated to $M$ is equivalent to $f$.
\end{Proposition}

\section*{Part II: Classification}

{

\section{Compatibility of skeleton}
\labell{sec:compatible skeleton}

Let $(M,\omega,\Phi,\calT)$ be a tall complexity one space.
The purpose of this section is to show that
the set $M_\exc$ of exceptional orbits is a tall skeleton over $\calT$,
the moment image $\Phi(M)$ is a convex Delzant subset of $\calT$,
and the set $\Phi(M)$ and skeleton $M_\exc$
are compatible.
See Definitions~\ref{def:skeleton}, \ref{def:delzant}, 
and~\ref{def:compatible}.

We begin with an important observation.

\begin{Lemma} \labell{nonexceptional Y}
Let $Y = T \times_H \C^{h+1} \times \h^0$ be a complexity one model in which 
$\{ [\lambda,0,0] \}$ is
a non-exceptional orbit.  Then, after possibly permuting the coordinates,
$Y = T \times_H \C^h \times \C \times \h^0$ 
and $H$ acts on $\C^h$ through an isomorphism with $(S^1)^h$.
Consequently, every orbit in $Y$ is non-exceptional,
and the image of $Y$ is a Delzant cone.
\end{Lemma}

\begin{proof}
Inside the model, the set of points that have stabilizer $H$ 
and that lie
in the same moment fiber as $[\lambda,0,0]$ 
is $T \times_H (\C^{h+1})^H \times \{ 0 \}$,
where $(\C^{h+1})^H$ is the subspace fixed by $H$.
Since $\{ [\lambda,0,0] \}$ is not an exceptional orbit,
this subspace is not trivial.
The result then follows from a dimension count.
\end{proof}

We can now prove the main results of this section.

\begin{Lemma} \labell{exceptional is closed}
Let $(M,\omega,\Phi,\calT)$ be a tall complexity one space.
The set $M_\exc$ of exceptional orbits,
labelled by the isotropy representations in $M$
and equipped with the map $\Phibar \colon M_\exc \to \t^*$ 
induced by the moment map, is a tall skeleton over $\calT$.
\end{Lemma}

\begin{proof}
By construction, $M_\exc$ satisfies the local requirement 
in the definition of a tall skeleton.
By the local normal form theorem and Lemma~\ref{nonexceptional Y},
the set of non-exceptional orbits is open, and hence $M_\exc$ is closed
in $M/T$.
So the restriction $\Phibar|_{M_\exc} \colon M_\exc \to \calT$ is proper.
\end{proof}

\begin{Lemma} \labell{properties of Y}
The moment image of a tall complexity one model is a Delzant cone.
\end{Lemma}

\begin{proof}
Let $Y = T \times_H \C^{h+1} \times \h^0$
be a tall complexity one model with moment map $\Phi_Y$,
and let $\alpha = \Phi_Y ([\lambda,0,0])$.
By Lemma~\ref{lemma:Y}, there exists a non-exceptional orbit $x$
in $\Phi_Y\inv(\alpha)$.
Let $Y_x$ be the corresponding complexity one model.
By Lemma~\ref{nonexceptional Y},
$\image \Phi_{Y_x}$ is a Delzant cone at~$\alpha$.
By the local normal form theorem
and the stability of the moment map for $Y_x$ and for $Y$
(see Section~\ref{sec:background}),
there exists a neighbourhood $U$ of $\alpha$ in $\calT$
such that $U \cap \image \Phi_Y = U \cap \image \Phi_{Y_x}$.
Because $\image \Phi_Y$ and $\image \Phi_{Y_x}$ are invariant 
under dilations about $\alpha$, this implies that they are equal.
\end{proof}

\begin{Lemma} \labell{compatible}
Let $(M,\omega,\Phi,\calT)$ be a tall complexity one space.
Then the moment image $\Phi(M)$ is a convex Delzant subset of $\calT$
that is compatible with the tall skeleton $M_\exc$.
\end{Lemma}

\begin{proof}
Because $\calT$ is convex, $M$ is connected, 
and $\Phi \colon M \to \calT$ is proper,
$\Phi(M)$ is a convex closed subset of $\calT$;
see Section~\ref{sec:background}.
Let $x$ be a T-orbit in $M$, let $Y$ be the corresponding model, 
and let $\alpha = \Phibar(x)$.
By the local normal form theorem, stability for the moment map on $Y$, 
and stability of the moment map on $M$, 
there exists a neighbourhood $U$ 
of $\alpha$ such that $\Phi(M) \cap U = \image \Phi_Y \cap U$.
The claim now follows from Corollary \ref{tall models} and
Lemma~\ref{properties of Y}. 
\end{proof} 

}

{
\section{\DH functions for tall complexity one models}
\labell{sec:DH for model}

{

The purpose of this section is to define the \DH functions
for truncations of a tall complexity one model
(Definition~\ref{DH truncation})
and to prove their basic properties
(Corollaries~\ref{exists function} and~\ref{difference is integral affine}).
These functions were used in Section~\ref{sec:intro}
to define \emph{compatibility} of a (Duistermaat-Heckman) function
and a tall skeleton; see Definition~\ref{compatible rho}.

\begin{Lemma} \labell{xi exist}
Let $Y = T \times_H \C^{h+1} \times \h^0$ be a tall complexity one model.
The torus $$ G := T \times_H (S^1)^{h+1}$$
acts faithfully on this model.
Let $i_T \colon T \to G$ denote the inclusion map.
Then there exists a unique $(h+1)$-tuple of non-negative integers
$(\xi_0,\ldots,\xi_h)$ such that the following sequence is 
well defined and exact:
\begin{gather} \labell{short exact G}
\begin{CD}
\{ 1 \} @>>> T @> i_T >> G @> P >> S^1 @>>> \{ 1 \} ,
\end{CD} \\
\mbox{where} \quad P([\lambda,a]) = a^\xi := \prod_{k=0}^h a_k^{\xi_k}. 
\nonumber
\end{gather}
The sequence
\begin{equation} \labell{short exact}
\begin{CD}
 \{ 1 \} @>>> H @> \chi >> (S^1)^{h+1} 
   @> a \mapsto \prod_{k=0}^h a_k^{\xi_k} >> S^1 @>>> \{ 1 \}
\end{CD}
\end{equation}
is also exact, where  
$\chi \colon H \to (S^1)^{h+1}$ is the embedding through which
$H$ acts on $\C^{h+1}$.
\end{Lemma}

\begin{proof}
Lemma~\ref{xi exist} follows from Lemmas 5.2, 5.3, and 5.8 of~\cite{locun}.
For completeness, we give a direct argument.

For any integers $\xi_0,\ldots,\xi_h$, the sequence~\eqref{short exact G}
is well defined and exact if and only if the sequence~\eqref{short exact}
is exact.  

Because the quotient $(S^1)^{h+1}/\chi(H)$ is a one dimensional
compact connected Lie group, there exist integers 
$\xi_0 , \dots , \xi_h$ such that
\eqref{short exact} is exact;
these integers are unique up to 
replacing $(\xi_0,\ldots,\xi_h)$
by $(-\xi_0,\ldots,-\xi_h)$.

Let $\eta_0,\ldots,\eta_h$ be the weights for the $H$ action on $\C^{h+1}$.
Differentiating the relation $\chi(h)^\xi = 1$
gives $\sum_{k=0}^h \xi_k \eta_k = 0$.

The quadratic moment map for the $H$ action on $\C^{h+1}$
is given by $\Phi_H(z) = \sum_{k=0}^h \pi |z_k|^2 \eta_k$.
Because $Y$ is tall, the level set $\Phi_H\inv(0)$ contains 
more than one orbit.
Hence, there exist complex numbers $z_0,\ldots,z_h$, not all zero,
such that $\sum_{k=0}^h \pi |z_k|^2 \eta_k = 0$.

Because the action is effective,
the space of solutions $(x_0,\dots,x_h)$ 
of the equation $\sum x_k \eta_k = 0$
is one dimensional. Hence, the previous two paragraphs imply that the vectors 
$(\xi_0,\ldots,\xi_h)$ and $(|z_0|^2,\ldots,|z_h|^2)$ are proportional.
So, after possibly replacing $(\xi_0,\ldots,\xi_h)$ 
by $(-\xi_0,\ldots,-\xi_h)$,
the integers $\xi_0,\ldots,\xi_h$ are all non-negative.
\end{proof}

\begin{Definition} \labell{def:J}
We call the map $P \colon G \to S^1$ described above
the \textbf{defining monomial}; cf.~\cite[Definition~5.12]{locun}.
A \textbf{complementary circle} to $T$ in $G$
is a homomorphism $J \colon S^1 \to G$ such that $P \circ J = \id_{S^1}$.
\end{Definition}

Let $\g_\Z$ denote the integral lattice in $\g$
and $\g_\Z^*$ the weight lattice in $\g^*$.  Thus,
$$\g_\Z \cong \Hom(S^1,G)
\quad \mbox{and} \quad
\g^*_\Z \cong \Hom(G,S^1).$$

\begin{Remark}\labell{rmk:J}
Complementary circles always exist.  To see this, note that
the short exact sequence~\eqref{short exact G} gives rise to a short exact
sequence of lattices, $\{ 0 \} \to \ell \to \g_\Z \to \Z \to \{ 0 \}$.
Any splitting of this sequence determines a complementary circle
$J \colon S^1 \to G$ to $T$ in $G$.
\end{Remark}

Let $(M,\omega,\Phi)$ be a Hamiltonian $T$-manifold,
and let $A \subset M$ be a measurable subset.
The \textbf{\DH measure for the restriction of~$\mathbf{\Phi}$ to~$\mathbf{A}$}
is the push-forward by the moment map of the restriction to $A$
of the Liouville measure; a real valued function 
on $\Phi(A)$ is the \textbf{\DH function} 
for this restriction if its product with Lebesgue measure is the \DH measure.
As before, it is almost unique; see the discussion
on page~\pageref{page:unique}.

\begin{Definition} \labell{DH truncation} 
Let $Y=T \times_H \C^{h+1} \times \fh^0$ be a tall complexity one model
with moment map $\Phi_Y \colon Y \to \t^*$.
A real valued function $\rho$ on a subset of $\t^*$
is the \textbf{\DH function for a truncation of the model} if
there exist 
a complementary circle $J$ to $T$ in $G$ and a positive number $\kappa$
such that $\rho$ is the \DH function
for the restriction of $\Phi_Y$
to the subset 
\begin{equation} \labell{YJkappa}
 Y_{J,\kappa} := \varphi_J\inv((-\infty,\kappa])
\end{equation}
of $Y$, where $\varphi_J \colon Y \to \R$
is the moment map for the resulting circle action on $Y$,
normalized by $\varphi_J([\lambda,0,0]) = 0$.
\end{Definition}

Let $Y$ be a tall complexity one space, 
and let $\alpha = \Phi_Y([\lambda,0,0])$.
We will show in Corollaries~\ref{exists function} 
and~\ref{difference is integral affine}
that there exist \DH functions
for truncations of the model $Y$,
that they are well defined and continuous on a neighbourhood
of $\alpha$ in $\image \Phi_Y$, 
and that the difference of every two such functions 
is equal to an integral affine function
on some neighbourhood of $\alpha$ in $\image \Phi_Y$.

\begin{Lemma} \labell{lemma:iso}
Let $Y = T \times_H \C^{h+1} \times \h^0$ be  a tall
complexity one model,
and let $\tPhi_Y \colon Y \to \g^*$ be a moment map
for the action of $G := T \times_H (S^1)^{h+1}$.
There exists a linear isomorphism
\begin{equation} \labell{iso}
\begin{CD}
 \g^*  @>>> \h^0 \times \R^{h+1} 
\end{CD}
\end{equation}
with the following properties.
\begin{enumerate}
\item
The following diagram commutes:
$$ \xymatrix{
 \g^* \ar[d]_{i_T^*} \ar[rrr]^{\eqref{iso}} 
       &&& \h^0 \times \R^{h+1} 
         \ar[d]^{(\nu, s) \mapsto (\chi^*(s),\nu)} \\
 \t^* \ar[rrr]^{\cong} &&& \h^* \times \h^0,
} $$
where the bottom isomorphism is induced by the inner product on $\t$
that we have chosen,
$i_T \colon T \to G$ is the inclusion map,
and $\chi \colon H \to (S^1)^{h+1}$ is the embedding
through which $H$ acts on $\C^{h+1}$.
\item
The composition
$$ \begin{CD}
 Y @> \tPhi_Y >> \g^* @> \eqref{iso} >> \h^0 \times \R^{h+1}
\end{CD}$$
has the form
\begin{equation} \labell{intermediate}
 [\lambda,z,\nu] \ \mapsto \  
 \big( \nu \, , \, 
        \textstyle \pi |z_0|^2 \, , \, \ldots , 
\textstyle \pi |z_h|^2 \big) + \text{constant}.
\end{equation}
\item
Let $\xi$ be the element of $\g_\Z^*$
that corresponds to the defining monomial $P \in \Hom(G,S^1)$,
and let
$\xi_0,\ldots,\xi_h$ be the exponents of the defining monomial.
Then the isomorphism~\eqref{iso} carries $\xi$ 
to $(0,(\xi_0,\ldots,\xi_h))$.
\end{enumerate}
\end{Lemma} 

\begin{proof}
Let $i_H \colon H \to T$ denote the inclusion map.
The torus $G$ is the quotient of $T \times (S^1)^{h+1}$
by the image of the $H$ under embedding $a \mapsto (i_H(a)\inv,\chi(a))$.
Hence,
\begin{equation} \labell{g star 1}
   \g^* = \big\{ (\gamma,s) \in \t^* \times \R^{h+1} \ \big| \ 
          i_H^* (\gamma) = \chi^*(s) \big\} .
\end{equation}
Under the identification of $\t^*$ with $\h^* \times \h^0$,
the space~$\g^*$ becomes further identified with 
\begin{equation} \labell{g star 2}
\big\{ (\beta,\nu,s) \in \h^* \times \h^0 \times \R^{h+1} \ \big| \ 
          \beta = \chi^*( s)  \big\} .
\end{equation}
Now consider the composition 
\begin{equation} \labell{composition}
\xymatrixcolsep{2.75pc}
\xymatrix{
  \g^*  \ar[r]^-{\text{ inclusion}} & \t^* \times \R^{h+1}
        \ar[r]^-{\cong} & \h^* \times \h^0 \times \R^{h+1}
        \ar[r]^-{\text{ projection }} & \h^0 \times \R^{h+1} . }
\end{equation}
The projection $(\beta,\nu,s) \mapsto (\nu,s)$ is a linear isomorphism
from the space~\eqref{g star 2} -- which is the image of $\g^*$
in $\h^* \times \h^0 \times \R^{h+1}$ --
to $\h^0 \times \R^{h+1}$.
This proves that the composition~\eqref{composition} is a linear isomorphism.

Moreover, if $(\beta,\nu,s)$ is in~\eqref{g star 2}
then $(\beta,\nu) = (\chi^*(s),\nu)$.  This gives~(1).

The moment map for the $T$ action on $Y$,
as a map to $\h^* \times \h^0$, has the form
$[\lambda,z,\nu] \mapsto (\Phi_H(z),\nu) + \constant$. 
The moment map for the $(S^1)^{h+1}$ action on $Y$
has the form 
$$[\lambda,z,\nu] \mapsto 
\left( \pi |z_0|^2, \ldots, \pi |z_h|^2 \right) + \constant.$$
Therefore, the moment map for the $G$ action on $Y$, as
a map to $\h^* \times \h^0 \times \R^{h+1}$, has the form
$$[\lambda,z,\nu] \mapsto 
\left( \Phi_H(z) , \nu , \pi |z_0|^2, \ldots, \pi |z_h|^2 \right) 
 + \constant.$$
This implies (2). 

Since the restriction of $P \in \Hom(G,S^1)$ to the subtorus $T$ of $G$ 
is trivial, and the restriction of $P$ to the subtorus $(S^1)^{h+1}$ of $G$ 
is the homomorphism $a \mapsto \prod_{k=0}^h a_k^{\xi_k}$,
the natural embedding of $\g^*$
into $\t^* \times \R^{h+1}$ carries $\xi$ to $(0,(\xi_0,\ldots,\xi_h))$.
Projecting to $\h^0 \times \R^{h+1}$, we get~(3).
\end{proof}

}

\begin{Lemma} \labell{sigma}
Let $Y = T \times_H \C^{h+1} \times \h^0$ be a tall complexity one model
with moment map $\Phi_Y \colon Y \to \t^*$.
Let $J \in \Hom(S^1,G)$ be a complementary circle to $T$
in $G := T \times_H (S^1)^{h+1}$,
and let $j$ be the corresponding element of $\g_\Z$. 
Let $\tPhi_Y \colon Y \to \g^*$ be the unique $G$ moment map
that satisfies $i_T^* \circ \tPhi_Y = \Phi_Y$
and 
$\protect{\langle \tPhi_Y ([\lambda,0,0]) , j \rangle = 0}$.
There exists a unique continuous map
$$ \sigma \colon \image \Phi_Y \to \g^* $$
with the following properties.
\begin{enumerate}
\item
$ i_T^*(\sigma(\beta) + t \xi ) = \beta 
\ \text{ for all } \ \beta \in \image \Phi_Y \ \text{and} \ t \in \R$.
\item
$ \image \tPhi_Y = \{ \sigma(\beta) + t \xi \ | \ 
   \beta \in \image \Phi_Y \text{ and } t \geq 0 \} . $
\item $\sigma\big(\Phi_Y([\lambda,0,0]) \big)
 = \tPhi_Y([\lambda,0,0]).$
\end{enumerate}
Here, $i_T \colon T \to G$ is the inclusion map, and 
$\xi \in \g^*_\Z$  corresponds to the defining monomial $P \in \Hom(G,S^1)$.
\end{Lemma}

\begin{proof}
Let $\chi \colon H \to (S^1)^{h+1}$ be the embedding
through which $H$ acts on $\C^{h+1}$.
Let $\xi_0,\dots,\xi_h$ be the (non-negative) exponents of the defining monomial.
By \eqref{short exact},
the level sets of the projection $\chi^* \colon \R^{h+1} \to \h^*$ 
are the lines $s + \R (\xi_0,\ldots,\xi_h)$.

After possibly reordering the coordinates,
we may assume that 
\begin{equation}\labell{xiprop}
\xi_k > 0 \mbox{ for }0 \leq k \leq h' \quad \mbox{and} \quad \xi_k = 0\mbox{ for }h' < k \leq h;
\end{equation}
  let $h'' = h-h'$.
Consider the subset $\del \R_+^{h'+1} \times \R_+^{h''}$
of $\R_+^{h+1}$, consisting of $(h+1)$-tuples of non-negative numbers
in which at least one of the first $h'+1$ entries is equal to zero.
Consider the map
\begin{equation} \labell{homeo'}
\del \R_+^{h'+1} \times \R_+^{h''}
\to  \chi^*\big(\R_+^{h+1}\big)  , \qquad s \mapsto \chi^*(s).
\end{equation}
This map is a bijection because,
by \eqref{xiprop},
the line $s + \R (\xi_0,\ldots,\xi_h)$
meets $\del \R_+^{h'+1} \times \R_+^{h''}$ exactly once
for each $s \in \R_+^{h+1}$.
Restricted to each closed facet, this map coincides with a linear isomorphism,
and hence is open as a map to its image.
It follows that the map \eqref{homeo'} is open.

Since the map \eqref{homeo'} is a homeomorphism, it has a continuous inverse
$$\sigma_H \colon \chi^*\big(\R_+^{h+1}\big) \to \R_+^{h+1}.$$
We claim that $\sigma_H$ has the following properties.
\begin{enumerate}
\item[($1'$)]
$\chi^*(\sigma_H(\beta) + t(\xi_0,\dots,\xi_h) ) = \beta$ for all $\beta \in 
\chi^*\big(\R_+^{h+1}\big)$ and $t \in \R$.
\item[($2'$)]
$\R_+^{h+1} = \{ \sigma_H(\beta) + t(\xi_0,\ldots,\xi_h)
 \ | \ \beta \in \chi^*\big(\R_+^{h+1}\big) \text{ and } t \geq 0 \}.$
\item[($3'$)]
$\sigma_H(0) = 0$.
\end{enumerate}
Properties ($1'$) and ($3'$) follow immediately from the definition of
$\sigma_H$.  To prove ($2'$),
consider $\beta \in \chi^*(\R_+^{h+1})$ and $t \in \R$.
By \eqref{xiprop}, the fact that $\sigma_h(\beta) \in \del \R_+^{h'+1} \times \R_+^{h''}$
implies that $\sigma_H(\beta) + t (\xi_0,\dots,\xi_h)$ lies in $\R_+^{h+1}$ exactly if $t \geq 0$.

We may assume without loss of generality that $\Phi_Y([\lambda,0,0]) = 0$.
Then the identification of $\ft^*$ with $\fh^* \times \fh^0$ carries $\image \Phi_Y$
 onto $\chi^*(\R_+^{h+1}) \times \fh^0$.  Define 
$\sigma \colon \image \Phi_Y \to \g^*$ so that the following diagram commutes:
$$ \xymatrix{
 \image \Phi_Y \ar[d]_{\sigma} \ar[rrr]^{\cong} 
       &&& \chi^*(\R^{h+1}_+) \times \h^0  
         \ar[d]^{(\beta,\nu) \mapsto (\nu,\sigma_H(\beta))} \\
 \g^* \ar[rrr]^{\eqref{iso}} &&& \h^0 \times \R^{h+1}\ ,
} $$
where \eqref{iso} is the isomorphism defined in Lemma~\ref{lemma:iso}.
By part (2) of that lemma, \eqref{iso} carries $\image \tPhi_Y$ 
onto $\fh^0 \times \R_+^{h+1}$.  
Claims (1), (2), and (3) then follow from ($1'$), ($2'$), and ($3'$), respectively,
by Parts (1) and (3) of Lemma~\ref{lemma:iso}.
\end{proof}

A lattice element is \emph{primitive} if it is not a multiple
of another lattice element by an integer that is greater than one. 
The \emph{rational length} of an interval $[x,y]$ with rational slope in $\g^*$ 
is equal to the positive number $k$ such that $y-x = k\xi$
where $\xi$ is a primitive lattice element.

\begin{Lemma} \labell{lemma:preDH}
Let $Y = T \times_H \C^{h+1} \times \h^0$
be a tall complexity one model with moment map 
$\Phi_Y \colon Y \to \t^*$.
Let $J \in \Hom(S^1,G)$ be 
a complementary circle to $T$ in $G:= T \times_H (S^1)^{h+1}$;
let $j$ be the corresponding element of $\g_\Z$; 
and let  $\varphi \colon Y \to \R$ be the moment map
for the resulting circle action, normalized by $\varphi([\lambda,0,0])=0$.
Then the $T \times S^1$ moment
map $(\Phi_Y,\varphi) \colon \Phi_Y \to \t^* \times \R$ is proper,
and each fiber contains at most one $T \times S^1$ orbit.
Moreover, if $\sigma \colon \image \Phi_Y \to \g^*$
is the map given in Lemma~\ref{sigma},
then
\begin{multline}\labell{splitimage}
\image \,  ( \Phi_Y, \varphi ) = \\
  \left\{ (\beta, s) \in \t^* \times \R   \mid 
   \beta \in \image \Phi_Y \text{ and }
   s \geq \left< \sigma(\beta) , j \right> 
   \right\}.
\end{multline}
\end{Lemma}

\begin{proof}
Let $i_T \colon T \hookrightarrow G$
be the inclusion map.
Let $\tPhi_Y \colon Y \to \g^*$ be a $G$ moment map, 
normalized so that $(i_T^*,j) \circ \tPhi_Y = (\Phi_Y,\varphi)$.
Note that, in particular, $\langle \tPhi_Y([\lambda,0,0],j \rangle =
\varphi([\lambda,0,0]) = 0$; cf.~Lemma~\ref{sigma}. 
By Lemma~\ref{lemma:iso} (2), $\tPhi_Y$ is proper and each fiber 
contains at most one $G$~orbit.

Let $\xi \in \g_\Z^*$ correspond to the defining monomial $P \in \Hom(G,S^1)$.
Because $P \circ J$ is the identity map, $\left< \xi , j \right> = 1$.  
Moreover, since the homomorphism $J$ splits the short exact sequence 
\eqref{short exact G}, the map
$(i_T ,J) \colon T \times S^1 \to G$ is an isomorphism of groups.
Hence, the induced map $(i_T^*,  j) \colon \g^* \to \t^* \times \R$ 
is a linear isomorphism.
Therefore, by the first paragraph, $(\Phi_Y,\varphi)$ is proper
and each fiber contains at most one $T \times S^1$ orbit.

Finally, \eqref{splitimage} follows easily from properties (1) and (2) 
of Lemma~\ref{sigma}.
\end{proof}

\begin{Lemma} \labell{lemma:DHfunction}
Let $Y = T \times_H \C^{h+1} \times \h^0$
be a tall complexity one model with moment map 
$\Phi_Y \colon Y \to \t^*$.
Let $J \in \Hom(S^1,G)$ be a complementary circle 
to $T$ in $G := T \times_H (S^1)^{h+1}$;
let $j$ be the corresponding element of $\g_\Z$; and 
let $\varphi \colon Y \to \R$ be the moment map
for the resulting circle action, normalized by $\varphi([\lambda,0,0])=0$.
Given $\kappa \in \R$,  define
$$Y_{J,\kappa} = \varphi\inv\big( (-\infty,\kappa] \big) 
                \subset Y .$$

\begin{enumerate}

\item
If $\sigma \colon \image \Phi_Y \to \g^*$
is the map given in Lemma~\ref{sigma},
then the function 
\begin{equation} \labell{formula}
\beta \mapsto 
   \kappa - \left< \sigma(\beta) , j \right> 
\end{equation}
from $\Phi_Y(Y_{J,\kappa})$ to $\R$
is a \DH function for the restriction of $\Phi_Y$ to~$Y_{J,\kappa}$.

\item
If $\kappa > 0$, then $\Phi_Y(Y_{J,\kappa})$ contains a neighbourhood
of  $\alpha = \Phi_Y([\lambda,0,0])$
in $\Phi_Y(Y)$.

\item
The restriction of $\Phi_Y$ to $Y_{J,\kappa}$ is proper.
\end{enumerate}
\end{Lemma}

\begin{proof}
Since by Lemma~\ref{lemma:preDH} each level set of $(\Phi_Y, \varphi)$ contains at most
a  single $T \times S^1$ orbit,
the \DH measure for  
the restriction of $(\Phi_Y,\varphi)$ to 
$Y_{J,\kappa}$ is Lebesgue measure
on the set $(\Phi_Y,\varphi)(Y_{J,\kappa})$.
The \DH measure 
for the restriction of $\Phi_Y$ to $Y_{J,\kappa}$ is the
push-forward of this measure under the projection map
from $\t^* \times  \R$  to $\t^*$.
Finally, by  \eqref{splitimage}, 
\begin{multline} \labell{displayed equation}
( \Phi_Y, \varphi )(Y_{J,\kappa}) = \\
  \left\{ (\beta, s) \in \t^* \times \R   \mid 
   \beta \in \image \Phi_Y \text{ and }
   \kappa \geq s \geq \left< \sigma(\beta) , j \right> 
   \right\}.
\end{multline}
Claim~(1) follows immediately.

By Lemma~\ref{sigma}, the function $\sigma$ is continuous and
$\langle \sigma(\alpha), j \rangle 
 = \langle \tPhi_Y([\lambda,0,0], j \rangle = 
 \varphi([\lambda,0,0]) = 0$.
Therefore, if $\kappa > 0$,    then 
the equation~\eqref{displayed equation} implies that 
there is a neighbourhood $U$ of $\alpha$ in $\t^*$ 
such that $\Phi_Y(Y) \cap U = \Phi_Y(Y_{J,\kappa}) \cap U$. 
This gives Claim~(2).

Since $\sigma$ is continuous, 
the equation~\eqref{displayed equation} implies that
the intersection $(K \times \R) \cap (\Phi_Y,\varphi)(Y_{J,\kappa})$
is compact for any compact set $K$.
Moreover, by Lemma~\ref{lemma:preDH}, $(\Phi_Y,\varphi)$ is proper.
Claim~(3) follows immediately.
\end{proof}

\begin{Corollary} \labell{exists function}
For every tall complexity one model $Y$, there exist \DH functions
for truncations of the model.
Each such function is
well defined and continuous on a neighbourhood of $\alpha$ in $\image \Phi_Y$.
\end{Corollary}

\begin{proof}
By Remark~\ref{rmk:J}, there exist complementary circles to $T$ in $G$.
The result then follows 
from Lemma~\ref{lemma:DHfunction} with any $\kappa > 0$. 
\end{proof}

\begin{Corollary} \labell{difference is integral affine}
Let $Y = T \times_H \C^{h+1} \times \h^0$ be a tall complexity one
model.  Let $\rho$ and $\rho'$ be \DH functions for truncations of $Y$.  
Then the difference $\rho - \rho'$ is equal to an integral affine function
on some neighbourhood of $\Phi_Y([\lambda,0,0])$ in $\image \Phi_Y$.
\end{Corollary}

\begin{proof}
Let $J$ and $J'$ be complementary circles to $T$ 
in $G:= T \times_H (S^1)^{h+1}$,
and let $j$ and $j'$ be the corresponding elements of $\g_\Z$.
Let $\varphi$ and $\varphi'$ be the associated moment maps,
normalized by $\varphi([\lambda,0,0]) = \varphi'([\lambda,0,0]) = 0$.
Let $P \colon G \to S^1$ be the defining monomial.
Because $P \circ J = P \circ J' = \text{Id}_{S^1}$, 
there exists $\Delta j \in \t_\Z$
such that 
\begin{equation*} 
j - j' = i_T(\Delta j).
\end{equation*}

Let $\sigma$ and $\sigma'$ be the maps from $\image \Phi_Y$ to $\g^*$
that are associated to $J$ and $J'$, respectively,
in Lemma~\ref{sigma}.
Let $\tPhi_Y$ and $\tPhi_Y'$ be the $G$ moment maps, normalized
as in Lemma~\ref{sigma}.  Because $i_T^* \circ \tPhi_Y = i_T^* \circ \tPhi_Y'$,
there exists a real number $c$ such that
$\tPhi_Y - \tPhi_Y' = c \xi$,
where $\xi \in \g_\Z^*$ corresponds to $P \in \Hom(P,S^1)$.

Therefore,  parts (1) and (2) of Lemma~\ref{sigma} 
imply that
$$\sigma - \sigma' = c \xi.$$

Let $\kappa$ and $\kappa'$ be positive numbers.
Let $\rho_{J,\kappa}$ and $\rho_{J',\kappa'}$ 
be the \DH functions for the restrictions of $\Phi_Y$
to $\varphi\inv((-\infty,\kappa])$ 
and ${\varphi'}\inv((-\infty,\kappa'])$, respectively.
By Lemma~\ref{lemma:DHfunction}, there exists 
neighbourhood of $\Phi_Y([\lambda,0,0])$ in $\image \Phi_Y$
where
\begin{multline*}
 \rho_{J,\kappa}(\beta) - \rho_{J',\kappa'}(\beta) \\
  = \kappa 
    -  \left< \sigma(\beta) , j \right>
          -\kappa'    + \left< \sigma'(\beta) , j' \right>  \\
  = \kappa 
    -  \left< \sigma(\beta) , j \right>
          -\kappa' +  \left< \sigma(\beta) - c \xi , j - i_T(\Delta j) \right>  \\
  = \kappa - \kappa' - c
    - \left< \iota_T^*(\sigma(\beta)) , \Delta j \right>  \\
  = \kappa - \kappa' - c
     - \left< \beta , \Delta j \right>. \\
\end{multline*}
Here, the penultimate equality uses the fact that $\left< \xi , j \right> = 1$,
and the last equality follows from Lemma~\ref{sigma}.
\end{proof}

{
\section{Compatibility of the Duistermaat-Heckman function, and local existence}
\labell{sec:DH compatible}

This section achieves two goals.  
In Proposition~\ref{DH is compatible} we prove
that a \DH function of a complexity one space
is compatible with its skeleton. 
In Proposition~\ref{local existence}
we prove ``local existence": 
for any compatible values of our invariants,
over sufficiently small open subsets of $\t^*$ there exists a 
tall complexity one space whose invariants take these values.
The proofs of both propositions
rely on a surgery that removes or adds exceptional orbits. 
In order to perform this surgery 
we identify a punctured neighbourhood of an exceptional orbit
with a punctured neighbourhood of a non-exceptional orbit.
This identification is done in Lemma~\ref{identify}.

We begin by setting up the relevant notation.
Let $C$ be a Delzant cone in $\t^*$, and let $(M_C,\omega_C,\Phi_C)$ 
be a symplectic toric manifold whose moment image is $C$.
We recall how to obtain such a manifold.
By Definition~\ref{def:delzant} 
there exist an integer $0 \leq k \leq n$
and a linear isomorphism $A \colon \R^n \to \ft^*$
that sends $\Z^n$ onto the weight lattice $\ell^*$ such that
$C = \alpha + A(\R_+^k \times \R^{n-k})$.
We may take $M_C$ to be the manifold
$$ M_C := \C^k \times (T^*S^1)^{n-k},$$
with the standard symplectic structure;
with the $T$-action given by the isomorphism
$T \to (S^1)^k \times (S^1)^{n-k}$ induced by
$A^* \colon \t \to \R^n$; 
and with the moment map
$\Phi_C (z,a,\eta) = \alpha +
 A \left( \pi |z_1|^2, \ldots, \pi |z_k|^2, \eta_1,\ldots,\eta_{n-k} 
   \right)$,
where $z = (z_1,\ldots,z_k) \in \C^k$
and $(a,\eta) \in (S^1)^{n-k} \times \R^{n-k} \cong (T^*S^1)^{n-k}$.

We also consider the manifold $M_C \times \C$
with the product symplectic structure.
This manifold admits a 
$T$ action  on the first factor with moment
map $(m,z) \mapsto \Phi_C(m)$ for all $m \in M_C$ and $z \in \C$.
It also admits a toric action of $T \times S^1$
with moment map $(m,z) \mapsto (\Phi_C(m) , \kappa+ \pi |z|^2)$,
for any $\kappa \in \R$. 

Given $\epsilon > 0$, let $D_\epsilon$ be the disk 
$$ D_\epsilon = \big\{ z \in \C \ \big| \ \pi |z|^2 <  \epsilon \big\}.$$}

\begin{Lemma} \labell{identify}
Let $Y = T \times_H \C^{h+1} \times \h^0$
be a tall complexity one model
with moment map $\Phi_Y \colon Y \to \t^*$.
Let $J$ be a complementary circle to $T$ 
in $G := T \times_H (S^1)^{h+1}$;
and let $\varphi \colon Y \to \R$ be the moment map 
for the resulting circle action, normalized by $\varphi([\lambda,0,0])=0$.
Let $\alpha = \Phi_Y([\lambda,0,0])$.
\begin{enumerate}
\item
Let $V$ be a neighbourhood of the orbit $\{[\lambda,0,0]\}$ in $Y$.
Then there exists a neighbourhood $U$ of $\alpha$ in $\t^*$
and a positive number $\kappa'$ such that the preimage
$\left( \Phi_Y , \varphi \right)\inv  
 \big( U \times (-\infty,\kappa') \big)$
is contained in $V$.

\item
Let $(M_C,\omega_C,\Phi_C)$ be a symplectic toric manifold
whose moment image is\footnote
   {The moment cone is a Delzant cone by Lemma~\ref{properties of Y}.} 
$C := \image \Phi_Y$.
For every positive number $\kappa$,
there exist $\eps > 0$,
a neighbourhood $U$ of $\alpha \in \t^*$,
and a $T$ equivariant symplectomorphism between
\begin{gather*}
\Phi_Y^{-1}(U) \cap \varphi^{-1}((\kappa,\kappa+\eps)) \subset Y 
\quad \mbox{and} \\
\Phi_C\inv(U) \times \left( D_\epsilon \smallsetminus \{0\} \right) \subset
M_C \times  \C
\end{gather*}
that intertwines  $\varphi \colon Y \to \R$ and the map
$(m,z) \mapsto \kappa + \pi |z|^2$.
Here, $T$ acts only on the first factor of  $M_C \times \C$.
\end{enumerate}
\end{Lemma}

\begin{proof}
Let $j$ be the element of $\Z_G \subset \g$ that corresponds
to $J \in \Hom(S^1,G)$.
Let $\sigma \colon \image \Phi_Y \to \g^*$ be given as in Lemma~\ref{sigma}.
Then by Lemma~\ref{sigma} and Lemma~\ref{lemma:preDH}:
\begin{itemize}
\item [(i)] the map $\sigma$ is continuous  and  
$\langle \sigma(\alpha), j \rangle = 0$;
\item [(ii)] $(\Phi_Y,\varphi) \colon Y \to \ft^* \times \R$ is proper and
each fiber contains at most one orbit; and
\item[(iii)]
$\image \,  ( \Phi_Y, \varphi ) =
  \left\{ (\beta, s) \in \t^* \times \R   \mid 
   \beta \in \Phi_Y(Y) \ \& \ 
   s \geq \left< \sigma(\beta) , j \right> 
   \right\}.$
\end{itemize}

Let $V$ be a neighbourhood of the orbit $\{[\lambda,0,0]\}$ in $Y$.  
By (ii) above, there exist a neighbourhood $U$ of $\alpha$ in $\t^*$ and 
a positive number $\kappa'$
such that the pre-image $(\Phi_Y,\varphi)\inv (U \times (-\kappa',\kappa') )$
is contained in $V$.
Therefore,  by  (i) and (iii) above,
after possibly shrinking $U$, 
the pre-image $(\Phi_Y,\varphi)\inv(U \times (-\infty,-\kappa'])$ is empty;
this proves Claim (1).

Moreover, if $\kappa$ is a positive number, then by  (i) and (iii)
above, there exists 
a neighbourhood $W$ of $(\alpha,\kappa) \in \t^* \times \R$
such that $\image (\Phi_Y,\varphi) \cap W = ( C \times \R)  \cap W$.
Consider $Y$ and $M_C \times S^1 \times \R$ as symplectic manifolds
with toric $T \times S^1$ actions and with moment maps
$(\Phi_Y,\varphi)$ and $(\Phi_C,\kappa+\proj_\R)$, respectively.
By the argument above, their moment images coincide with $C \times \R$,
hence with each other, on a neighbourhood of $(\alpha,\kappa)$.
Because (ii) holds,
the local normal form for toric actions (see~\cite{delzant}) implies that,
after possibly shrinking this neighbourhood,
its preimages in $Y$ and in $M_C \times S^1 \times \R$ are isomorphic.
Thus, after possibly shrinking $U$, for sufficiently small $\eps$
the subset $\Phi_Y\inv(U) \cap \varphi\inv((\kappa-\eps,\kappa+\eps))$
of $Y$ is isomorphic to the subset 
$\Phi_C\inv(U) \times S^1 \times (\kappa-\eps,\kappa+\eps)$
of $M_C \times S^1 \times \R$.
Restricting to the preimage of $U \times (\kappa,\kappa+\eps)$,
and composing further with the map
$re^{i\theta} \mapsto (e^{i\theta} , \kappa + \pi r^2)$,
which is an $S^1$-equivariant symplectomorphism 
from $D_\eps \ssminus \{ 0 \}$ onto $S^1 \times (\kappa,\kappa+\eps)$
that carries the map $z \mapsto \kappa + \pi |z|^2$
to the map $(\lambda,s) \mapsto s$, we get Claim~(2).
\end{proof}

We are now ready to prove the first main result of this section.

\begin{Proposition} \labell{DH is compatible}
Let $(M,\omega,\Phi,\calT)$ be a tall complexity one space.
The skeleton and \DH function of $M$ are compatible.
\end{Proposition}

\begin{proof}
Fix a point $\alpha$ in the moment image $\Delta = \Phi(M)$.
By Lemma~\ref{exceptional is closed} and Corollary~\ref{discrete},
there are only finitely many exceptional orbits $x$ in $\Phi\inv(\alpha)$.
For each such $x$, 
let $Y_x = T \times_{H_x} \times \C^{h_x + 1} \times \fh_x^0$ be
the tall complexity one model with moment map 
$\Phi_x \colon Y_x \to \ft^*$ corresponding to $x$.

By applying surgery to neighbourhoods of the exceptional orbits, 
we will construct a complexity one space $(M',\omega',\Phi',U)$  with
moment image $\Delta \cap U$ that has no exceptional orbits,
where $U \subset \calT$ is a convex open neighbourhood of $\alpha$.
Additionally, the \DH function $\rho' \colon \Delta \cap U \to \R_{>0}$ of
$M'$ will satisfy
\begin{equation}\labell{DHcompeq}
\rho' = \rho - \sum_x \rho_x,
\end{equation}
where 
$\rho_x$ is the \DH function for a truncation of the model  $Y_x$  
for each  exceptional orbit $x$ in $\Phi\inv(\alpha)$, 
and the sum is over all such orbits.
Because $M'$ is a complexity one space with no exceptional orbits,
$\rho'$ is integral affine on $\Delta \cap U$. 
The proposition will follow immediately.

By Lemma~\ref{compatible} 
there exists a Delzant cone $C$ at $\alpha$
and  a convex open neighbourhood $U$ of $\alpha$ in $\t^*$
such that $\Delta \cap U = C \cap U$.
Let $(M_{C},\omega_C,\Phi_{C})$ be a toric manifold with moment image $C$.

By the local normal form theorem, for each exceptional orbit $x$
in $\Phi\inv(\alpha)$ there exists an isomorphism
$\Psi_x$ from an invariant neighbourhood $V_x$ 
of the orbit $\{[\lambda,0,0]\}$  in $Y_x$
to an invariant open subset of $M$ 
that carries  $\{[\lambda,0,0]\}$ to $x$. 
Moreover, we may assume that
the closures in $M$ of the open subsets $\Psi_x(V_x)$ are disjoint.

Given an exceptional orbit $x$ in $\Phi\inv(\alpha)$,
let $J_x$  be a complementary circle to $T$ in $T \times_{H_x} (S^1)^{h_x+1}$;
see Remark~\ref{rmk:J}.
Let $\varphi_x$ be a moment map for the resulting circle action,
normalized by $\varphi_x([\lambda,0,0])=0$.
By  Lemma~\ref{compatible}, $\image \Phi_x = C$.
By Lemma~\ref{identify}, for sufficiently small $\epsilon > 0$,
after possibly shrinking $U$,  there exists $\kappa_x > 0$
such that $\Phi_{x}\inv(U) \cap \varphi_x\inv((-\infty,\kappa_x+\eps))$
is contained in $V_x$;
moreover,
there exists an isomorphism 
between
\begin{gather}
\labell{iso1}
\Phi\inv(U) \cap \Psi_x\big(\varphi_x\inv((\kappa_x,\kappa_x+\epsilon))\big)
\subset M \quad \mbox{and} \\
\labell{iso2}
  \Phi_{C}\inv(U) \times \big( D_\epsilon \smallsetminus \{0\} \big)
 \subset M_C \times \C
\end{gather}
that intertwines
$\varphi_x \circ \Psi_x\inv \colon \Psi_x(V_x) \to \R$ and the map
$(m,z) \mapsto \kappa_x + \pi |z|^2.$

We construct $M'$ by gluing together the spaces
\begin{gather}\labell{space1}
\Phi\inv(U)  \smallsetminus 
\bigsqcup_x  \Psi_x\big(\varphi_x\inv(-\infty, \kappa_x]\big)\  \subset M \quad
\mbox{and} \\ \labell{space2}
\bigsqcup_{x} \Phi_{C}\inv(U) \times D_\eps \ \subset M_C \times \C
\end{gather}  by identifying
their isomorphic open subsets \eqref{iso1} and \eqref{iso2} for every exceptional orbit $x$ in $\Phi\inv(\alpha)$.

The space $M'$ is
the union two {\em closed} subspaces: the images in $M'$ of
\begin{gather}\labell{closed1}
\Phi\inv(U)  \smallsetminus 
\bigsqcup_x \Psi_x\big(\varphi_x\inv(-\infty, \kappa_x + \tfrac{1}{2} \epsilon) \big)
\ \subset M \quad \mbox{and} \\
\labell{closed2}
\bigsqcup_x \Phi_{C}\inv(U) \times \ol{D}_{\frac{1}{2}\eps}\  \subset M_C \times \C,
\end{gather}
where $\ol{D}_{\frac12\eps} \subset \C$ is the closed disk.
Because each of these is Hausdorff,
$M'$ is Hausdorff.

By construction, $M'$ is a manifold with a $T$ action, 
a symplectic form $\omega'$,
and a moment map $\Phi'$.
Moreover,  as maps to $U$,
the restriction of $\Phi$ to \eqref{closed1}
and the restriction of $\Phi_C$ to \eqref{closed2}
are both proper, and so $\Phi'$ is proper as well.

This yields a complexity one space $(M',\omega',\Phi',U)$ with moment image $\Delta \cap U$ that
has no exceptional orbits in ${\Phi'}\inv(\alpha)$. 
By Lemma~\ref{exceptional is closed}, the restriction of $\Phi'$ to the set of exceptional orbits is proper.
Therefore, after possibly shrinking $U$
further, we may assume that $M'$ has no exceptional orbits.

By Lemma~\ref{lemma:DHfunction}, after possibly shrinking $U$,
there is a well-defined \DH function 
$\rho_x \colon \Delta \cap U \to \R_{>0}$
for the restriction of $\Phi_x$ to $\varphi_x^{-1}((-\infty,\kappa_x])$,
and thus for the restriction of $\Phi$ 
to $\Psi_x(\varphi_x^{-1}((-\infty,\kappa_x])$.
Since $M'$ and \eqref{space1} differ by a set of measure zero,
and since $\rho_x$ is a \DH function for a truncation of the model $Y_x$,
the \DH function $\rho'$ for $(M',\omega',\Phi')$ satisfies \eqref{DHcompeq},
as required.

\end{proof}

We proceed to the second main result of this section:

\begin{Proposition}[Local existence] \labell{local existence}
Let $(S,\pi)$ be a tall skeleton over an open subset
$\calT \subset \t^*$,
let $\Delta \subset \calT$ be a convex Delzant subset that is compatible 
with $(S,\pi)$,
and let $g$ be a non-negative integer.
\begin{enumerate}
\item
For any $\alpha \in \calT$
there exists a convex open neighbourhood $U \subset \calT$ of $\alpha$
and a tall complexity one space of genus $g$ over $U$
with moment image $\Delta \cap U$  whose skeleton
is isomorphic to $S|_U$.
\item
Let $\rho \colon \Delta \to \R_{>0}$ be a function that is
compatible with $(S,\pi)$.
Then for any $\alpha \in \calT$
there exists a convex open neighbourhood $U \subset \calT$ of $\alpha$
and a tall complexity one space of genus $g$ over $U$
with moment image  $\Delta \cap U$,
whose skeleton is isomorphic to $S|_U$
and whose \DH function is $\rho|_{\Delta \cap U}$.
\end{enumerate}
\end{Proposition}

\begin{proof}[Proof of part (1)]
Let $\alpha$ be a point in $\calT$
and fix an arbitrary positive number $b$.
By Corollary~\ref{discrete},
the level set $\pi\inv(\alpha)$ in $S$ is finite.
We will choose a convex open neighbourhood $U \subset \calT$ of $\alpha$
and construct a complexity one space $(M',\omega',\Phi',U)$
of genus $g$ 
with moment image $\Delta \cap U$
whose skeleton is isomorphic to $S|_U$.
Additionally, the  \DH function $\rho'$ of $M'$ will be smaller
than $b$ near $\alpha$, and will satisfy
\begin{equation} \labell{DH of Mprime}
\rho' = c + \sum_{s \in \pi\inv(\alpha)} \rho_s,
\end{equation}
where $c$ is a positive real number and where
each $\rho_s$ is the \DH function for a truncation of the model $Y_s$ 
associated to $s \in S$.

Choose a closed symplectic 2-manifold $(\Sigma, \eta)$ of genus $g$
and a positive number $\kappa_s$
for each $s \in \pi\inv(\alpha)$ such that
$$\int_\Sigma \eta + \sum_{s \in \pi\inv(\alpha)} \kappa_s < b.$$

Since $\Delta$ is a Delzant subset, there exists
a convex open neighbourhood $U$ of $\alpha$ and a Delzant cone $C$ at $\alpha$
such that $\Delta \cap U = C \cap U$.
Let $(M_C,\omega_C, \Phi_C)$ be a toric manifold
with moment image $C$.

By the Darboux theorem, if $\eps > 0$ is sufficiently small,
for each $x \in \pi\inv(\alpha)$ we can choose an open set $W_{s}$ in $\Sigma$
and a symplectomorphism $\Psi_s \colon D_\eps \to W_{s}$;
let $x_s = \Psi_s(0)$.
Furthermore, we may assume  that the closures 
of the sets $W_{s}$ are disjoint. 
Clearly, for each $s \in \pi\inv(\alpha)$,
there is an isomorphism
between
\begin{gather}
 \Phi_C\inv(U) \times \left( D_\epsilon \smallsetminus \{0\} \right) 
 \ \subset \ M_C \times  \C
\quad \mbox{and}
\\
\labell{open subset}
 \Phi_C \inv(U) \times \left( W_{s} \ssminus \{x_s \} \right)
\ \subset \ M_C \times \Sigma,
\end{gather}
given by $(m,z) \mapsto (m,\Psi_s(z))$.

Let $Y_s = T \times_{H_s} \C^{h_s+1} \times \h_s^0$ be the
tall complexity one model with moment map
$\Phi_s \colon Y_s \to \ft^*$
corresponding to $s \in \pi\inv(\alpha)$.
Let $J_s$ be a complementary circle to $T$ in $T \times_{H_s} (S^1)^{h_s+1}$;
see Remark~\ref{rmk:J}.
Let $\varphi_s$ be a moment map for the resulting circle action, 
normalized by $\varphi_s([\lambda,0,0])=0$.
Since $(S,\pi)$ and $\Delta$ are compatible, $\image \Phi_{Y_s} = C$.
By part (2) of Lemma~\ref{identify}, after possibly shrinking 
$U$ and $\epsilon$,
there exists an isomorphism between 
\begin{gather}
\labell{open subset2}
\Phi_{s} \inv(U) \cap \varphi_s\inv\left((\kappa_s,\kappa_s+\eps)\right)
\ \subset \ Y_s  \quad  \mbox{and} \\
\Phi_C\inv(U) \times  \left( D_\epsilon \smallsetminus \{0\} \right)
\ \subset \ M_C \times \C
\end{gather}
that intertwines $\varphi_s \colon Y_s \to \R$ and the map 
$(m,z) \mapsto \kappa_s + \pi \left|z\right|^2$.

Because the sets~\eqref{open subset} and~\eqref{open subset2}
are both isomorphic to the same set,
there exists an isomorphism between them
that intertwines the map $(m,\Psi_s(z)) \mapsto \kappa_s+ \pi |z|^2$
and the map $\varphi_s$.
We construct $M'$ by gluing together the spaces
\begin{gather} \labell{glue1}
 \Phi_C\inv(U) \times
   \left(\Sigma \ssminus \{ x_s \}_{s \in \pi\inv(\alpha)} \right)
    \ \subset \ M_C \times \Sigma \quad \ \ \mbox{and} \\
\bigsqcup_{s \in \pi\inv(\alpha)}
\Phi_{s}\inv(U) \cap \varphi_s\inv((-\infty,\kappa_s+\epsilon))
\ \subset \ \bigsqcup_{s \in \pi\inv(\alpha)} Y_s 
\end{gather}
by identifying their isomorphic open subsets~\eqref{open subset}
and~\eqref{open subset2} for every $s \in \pi\inv(\alpha)$.

The space $M'$ is the union of two {\em closed} subspaces:
the images in $M'$ of
\begin{gather} \labell{thing1}
  \Phi_C\inv(U) \times \Big( \Sigma \ssminus 
   \bigsqcup_{s \in \pi\inv(\alpha)} \Psi_s(D_{\eps/2}) \Big) 
   \ \subset \ M_C \times \Sigma
\quad \ \ \mbox{and} \\ \labell{thing2}
   \bigsqcup_{s \in \pi\inv(\alpha)} \Phi_{s}\inv(U) \cap 
\varphi_s\inv \left( (-\infty , \kappa_s + \eps/2 ] \right)
   \ \subset \ \bigsqcup_{s \in \pi\inv(\alpha)} Y_s. 
\end{gather}
Because each of these is Hausdorff, $M'$ is Hausdorff.

By construction, $M'$ is a manifold with a $T$ action,
a symplectic form $\omega'$, and a moment map $\Phi'$.
Moreover, as a map to $U$,
the restriction of $\Phi_C$ to \eqref{thing1} is proper,
and 
the restriction of each $\Phi_s$ to \eqref{thing2}
is proper by part (3) of Lemma~\ref{lemma:DHfunction}.
Therefore,  $\Phi'$ is proper as a map to~$U$.

This yields a complexity one space $(M',\omega',\Phi',U)$
with moment image $\Delta \cap U$,
and an isomorphism from $M'_\exc$
onto a neighbourhood of $\pi\inv(\alpha)$ in $S$.
Because $\pi$ is proper, after possibly shrinking $U$ further, 
we may assume that the image of $M'_\exc$ in $S$ is $S|_U$.
Since it is straightforward to check that $M'$ has genus $g$,
it remains to show that the \DH function $\rho'$ of $M'$
is smaller than $b$ near $\alpha$.

$M'$ can be written as the disjoint union of 
\eqref{glue1} and
$$
\bigsqcup_{s \in \pi\inv(\alpha)}
\Phi_{s}\inv(U) \cap \varphi_s\inv((-\infty,\kappa_s])
\ \subset \ \bigsqcup_{s \in \pi\inv(\alpha)} Y_s. $$
Clearly, the function that takes the constant value $c = \int_\Sigma \eta$  
on $C \cap U$ is a \DH function for the set \eqref{glue1}.
Moreover, by Lemma~\ref{lemma:DHfunction}, after possibly shrinking $U$,
for each $s \in \pi\inv(\alpha)$,
the function $\rho_s \colon \Delta \cap U \to \R$
given by
$\rho_s(\alpha) = \kappa_s - \langle \sigma(\alpha),j \rangle$
is a \DH function 
for the restriction of $\Phi_s$ to 
$\Phi_s^{-1}(U) \cap \varphi_s^{-1}((-\infty,\kappa_s]).$
In particular, it is a \DH function for a truncation of the model $Y_s$.
Thus, the \DH function $\rho'$ for $M'$ satisfies \eqref{DH of Mprime}.

Since $\rho_s(\alpha) = \kappa_s - \langle \sigma(\alpha),j \rangle = 
\kappa_s$, we
have chosen the positive numbers $\eta$ and $\kappa_s$ 
so that $\rho'$ is smaller than $b$ near $\alpha$.
\end{proof}

\begin{proof}[Proof of part (2)]
By part (1) there exists a complexity one space
$(M',\omega',\Phi',U)$ of genus $g$ with moment image $\Delta \cap U$
whose skeleton is isomorphic to $S|_U$ such that the \DH function $\rho'$ 
of $M'$ satisfies $\rho'(\alpha) < \rho(\alpha)$ and \eqref{DH of Mprime}.
Thus, since  $\rho$ is compatible with  $(S,\pi)$, 
Corollary~\ref{difference is integral affine}
implies that
$\rho - \rho'$ is integral affine on a neighbourhood of $\alpha$ in $\Delta$.
(Alternatively, this follows from Proposition~\ref{DH is compatible}.)
Thus, 
since $\rho' < \rho$ near $\alpha$,
there exists
 $\kappa > 0$ and $\zeta \in \ell \subset \t$ such that 
$\rho(\beta) - \rho'(\beta) = \kappa - \left< \beta-\alpha , \zeta \right>$.

The tall complexity one Hamiltonian $T$-manifold $Y = M_C \times \C$
with moment map $\Phi_Y(m,z) = \Phi_C(m)$ is isomorphic
to the tall complexity one model
$T \times_H (\C^k \times \C) \times \h^0$,
where $T \cong (S^1)^n$, where $H \cong \{ 1 \}^{n-k} \times (S^1)^k$,
and where $\h^0 \cong \R^{n-k}$.
So we can apply Lemma~\ref{lemma:DHfunction} to it.
We identify the corresponding torus $G = T \times_H (S^1)^{k+1}$ 
with $T \times S^1$,  and we identify
$[\lambda,0,0]$ with the point $(m_0,0)$ such that $\Phi_C(m_0)=\alpha$.

Define $J \colon S^1 \to T \times S^1$ 
by $J(\lambda) = (\lambda^\zeta , \lambda)$.
Let $\varphi \colon Y \to \R$ be the moment map for the resulting  circle
action, 
normalized by $\varphi(m_0,0)=0$.
The $T \times S^1$ moment map $\tPhi_Y \colon Y \to \ft^* \times \R$
given by 
$\tPhi_Y(m,z) = (\Phi_C(m), \pi |z|^2 - \left< \alpha , \zeta \right> )$
is normalized as in Lemma \ref{sigma}.
Its moment image is 
$\tPhi_Y(Y) = C \times \left[-\left< \alpha , \zeta \right> , \infty\right) $.
Hence, the map $\sigma \colon C \to \t^* \times \R$ 
described in Lemma~\ref{sigma}
is $\sigma(\beta) = (\beta, - \left< \alpha,\zeta \right> )$.
Therefore, by Lemma~\ref{lemma:DHfunction}, the restriction of $\Phi_Y$ 
to $\varphi\inv((-\infty,\kappa])$
is proper and the \DH function for this restriction is
$\rho_{J,\kappa}(\beta) = \kappa - \left<  \beta - \alpha, \zeta \right>$.

By part (2) of Lemma~\ref{identify}, there exists a $T$ equivariant
symplectomorphism from
$\Phi_Y\inv(U) \cap \varphi\inv((\kappa,\kappa+\eps))$
to $\Phi_C\inv(U) \times \left( D_\eps \ssminus \{ 0 \} \right)$
that carries the map $\varphi$ to the map
$(m,z) \mapsto \kappa + \pi |z|^2$.
If we glue this local model into $M'$ following the
same procedure as explained in the above proof of part (1), 
we get a new complexity one space, which satisfies all our requirements.
\end{proof}

\section{Proof of the existence theorems}
\labell{sec:proof}

We are now ready to prove the existence theorems
that we stated in Section~\ref{sec:intro}.

\begin{proof}[Proof of Theorem \ref{existence}]
By part (2) of Proposition \ref{local existence} for 
each point in
$\calT$ there exists a convex open subset $U \subset \calT$
containing the point and there exists
a complexity one space 
of genus $g:=\text{genus}(\Sigma)$ over $U$
whose moment image is $\Delta \cap U$,
whose skeleton is isomorphic to $S|_U$,
and whose \DH function is $\rho|_U$.
The result now follows from Proposition~\ref{combined}.
\end{proof}

\begin{proof}[Proof of Theorem \ref{data from manifold}]
This theorem is an immediate consequence of
Lemma~\ref{compatible},
Proposition \ref{DH is compatible}, and Theorem~\ref{existence}.
\end{proof}

\begin{proof}[Proof of Theorem \ref{thm:existence}]
This theorem is an immediate consequence of 
Lemma~\ref{DH existence} 
below
and Theorem~\ref{existence}.
\end{proof}

\begin{Lemma} \labell{DH existence}
Let $\calT \subset \t^*$ be an open subset.
Let $\Delta \subset \calT$ be a convex Delzant subset.
Let $(S,\pi)$ be a skeleton over $\calT$. 
If $\Delta$ and $(S,\pi)$ are compatible,
then there exists a function $\rho \colon \Delta \to \R_{>0}$
that is compatible with $(S,\pi)$.
\end{Lemma}

\begin{proof}
By part 1 of Proposition \ref{local existence}, there exists a cover
$\fU$ of $\calT$ by convex sets and, for each $U \in \fU$, a complexity
one space $(M_U,\omega_U,\Phi_U,U)$ whose moment image
is $U \cap \Delta$ and whose skeleton is $S|_U$.
Let $$\rho_U \colon U \cap \Delta  \to \R$$
be its \DH function.
By Proposition \ref{DH is compatible}, $\rho_U$ is 
compatible with the moment image $\Delta \cap U$ and the skeleton
$(S|_U,\pi_U)$.
Hence, by Definition \ref{compatible rho}, 
and Corollary~\ref{difference is integral affine},
on every intersection
the difference
$$ \rho_U|_{U \cap V} - \rho_V|_{U \cap V} $$
is locally an integral affine function.
Hence this difference is given by a locally constant function
$$ h_{U V} \colon U \cap V \to \ell \oplus \R .$$
Let $\cA$ denote the sheaf of locally constant functions
to $\ell \oplus \R$.
Because $\Delta$ is convex, the \v Cech cohomology
$H^1(\Delta,\cA)$ is trivial.  Hence, after possibly
passing to a refinement of the cover, there exist locally constant
functions 
\begin{equation} \labell{hU}
h_U \colon U \to \ell \oplus \R
\end{equation}
such that $h_U - h_V = h_{UV}$ on $U \cap V$.
Therefore, we can define
$\rho \colon \Delta \to \R$
by $\rho|_U = \rho_U - h_U$ for all $U \in \fU$,
where $h_U$ also denotes
the integral affine function given by \eqref{hU}.
\end{proof}


\end{document}